\documentclass[a4paper,12pt]{amsart}
\usepackage[utf8x]{inputenc}
\usepackage{mathrsfs}
\usepackage{amsfonts,amssymb,amscd,amsthm,amsmath,graphicx}
\usepackage[all]{xy}
\usepackage{color}

\theoremstyle{plain}
\newtheorem{Theorem}{Theorem}[section]
\newtheorem{Lemma}[Theorem]{Lemma}
\newtheorem{Corollary}[Theorem]{Corollary}
\newtheorem{Proposition}[Theorem]{Proposition}

\newtheorem{Definition}{Definition}[section]
\newtheorem{Remark}{Remark}[section]

\newtheorem{Example}{Example}[section]

\begin{document}

\title{A generalization of Duflo's conjecture}
 
\author{Hongfeng Zhang}

\address[Hongfeng Zhang]{BICMR, Peking University, No. 5 Yiheyuan Road, Haidian District, Beijing 100871, China.}
\email{zhanghongf@pku.edu.cn}

\abstract{In this article, we generalize Duflo's conjecture to understand the branching laws of non-discrete series.
We give a unified description on the geometric side about the restriction of an irreducible unitary representation $\pi$ of $\mathrm{GL}_n(k)$, $k=\mathbb{R}$ or $\mathbb{C}$, to the mirabolic subgroup, where $\pi$ is attached to a certain kind of coadjoint orbit.}
\endabstract

\maketitle



\tableofcontents

\section{Introduction}

One goal of the orbit method is to establish a correspondence between coadjoint orbits and irreducible unitary representations of Lie groups. One can benefit a lot from such a correspondence, for example, finding out  irreducible unitary representations through coadjoint orbits; or predicting the decomposition of $\pi|_H$ by looking at the decomposition of $\mathrm{p}(\mathcal{O}_{\pi})$
as $H$-orbit, where $\pi$ is a unitary representation of the Lie group $G$ and is attached to the coadjoint orbit $\mathcal{O}_{\pi}$, $H$ is a closed subgroup of $G$,  and $\mathrm{p}$ is the moment map from $\mathcal{O}_{\pi}$ to $\mathfrak{h}^*$ with $\mathfrak{h}^*$ the dual of the Lie algebra of $H$; etc.(see Kirillov \cite{K}, Introduction).

The orbit method is brought up by Kirillov, who attached an irreducible unitary representation to a coadjoint
orbit for nilpotent groups in a perfect way. Later, Kostant (see Auslander-Kostant \cite{AK}) established the theory of quantization, and built up the orbit method for solvable groups. Duflo \cite{D} constructed all irreducible unitary representations of almost algebraic groups from certain coadjoint orbits, assuming one knew unitary dual of reductive groups. It suggested that to built up the orbit method for general Lie groups, one should concentrate on reductive groups. The quantization problem of coadjoint orbits of reductive groups can be reduced to quantizing nilpotent coadjoint orbits. It has not been solved completely, but there are some wonderful results (see Vogan \cite{V3}, Chapter 10).

In spirit by some work on how the branching laws behave in the orbit method, including Heckman \cite{H} and Guillemin-Sternberg \cite{GS} for compact groups, Fujiwara \cite{F} for exponential solvable groups, and the work of Kobayashi \cite{Ko} on the branching laws of reductive groups, Duflo \cite{DV} formulated a conjecture for the branching problem of discrete series of almost algebraic group (see also Liu \cite{L}).

Let $G$ be an (almost) algebraic group, let $H$ be a closed algebraic subgroup of $G$. Let $\mathfrak{g}$ (resp. $\mathfrak{h}$) be the Lie algebra of $G$ (resp. of $H$). Let $\mathfrak{g}^*$ (resp. $\mathfrak{h}^*$) be the dual of $\mathfrak{g}$ (resp. of $\mathfrak{h}$). Let $\pi$ be a discrete series representation of $G$. By Duflo \cite{D1}, $\pi$ is attached to a strongly regular $G$-coadjoint orbit $\mathcal{O}_{\pi}$. Consider the restriction of $\pi$ to $H$, denoted by $\pi|_H$, Duflo's conjecture says that
\begin{itemize}
\item[(i)] $\pi|_H$ is $H$-admissible if and only if the moment map $\mathrm{p}:\mathcal{O}_{\pi}\to \mathfrak{h}^*$ is weakly proper.
\item[(ii)] If $\pi|_H$ is $H$-admissible, then each irreducible $H$-representation $\sigma$ which appears in $\pi|_H$ is attached to a strongly regular $H$-coadjoint orbit $\Omega$ which is contained in $\mathrm{p}(\mathcal{O}_{\pi})$.
\item[(iii)] If $\pi|_H$ is $H$-admissible, then the multiplicity of each such $\sigma$ can be expressed geometrically in terms of the reduced space of $\Omega$ with respect to the moment map $\mathrm{p}$.
\end{itemize}

Let's explain in details.  The notion ``almost algebraic
group" is explained in \cite{D1}. An element $f\in \mathfrak{g}^*$ is called strongly regular if $f$ is regular (i.e. the coadjoint orbit containing $f$ is of maximal dimension) and its ``reductive factor" $\mathfrak{s}(f):=\{X\in \mathfrak{g}(f): \mathrm{ad}X\ \text{is semisimple}\}$ is of maximal dimension among the reductive factors of all the regular elements in $\mathfrak{g}^*$, where $\mathfrak{g}(f)$ denotes the centralizer of $f$ in $\mathfrak{g}$.

Let $\mathcal{O}$ be a $G$-coadjoint orbit in $\mathfrak{g}^*$. Then $\mathcal{O}$ is equipped with the Kirillov-Kostant-Souriau symplectic form $\omega$ and becomes an $H$-Hamiltonian space. The corresponding moment map is the natural projection $\mathrm{p}: \mathcal{O}\to \mathfrak{h}^*$.

In (i), the notion ``$H$-admissible" is due to Kobayashi, which means that $\pi|_H$ is discretely decomposable (i.e. $\pi|_H$ is a direct sum of irreducible unitary representations of $H$) and all irreducible representations of $H$ has only finite multiplicities in $\pi|_H$.
The notion ``weak properness" means that the preimage (for $\mathrm{p}$) of each compact subset which is contained in $\mathrm{p}(\mathcal{O}_{\pi})\cap \Upsilon_{sr}$ is compact in $\mathcal{O}_{\pi}$, where $\Upsilon_{sr}$ denotes the set of strongly regular elements in $\mathfrak{h}^*$.

In (iii), the reduce space of $\Omega$ is equal to $\mathrm{p}^{-1}(\Omega)/H$.

Duflo's conjecture has been proved in some cases (see Liu \cite{L}), and is also generalized by Liu-Yu \cite{LY},  who verified Duflo's conjecture for the restriction of tempered representations of $\mathrm{GL}_n(k)$ (for $k=\mathbb{R}$ or $\mathbb{C}$) to its mirabolic subgroup and gave a geometric interpretation of the Kirillov's conjecture. Recently, Liu-Oshima-Yu \cite{LOY} verifies Duflo's conjecture for the restriction of irreducible unitary representations of $\mathrm{Spin}(N,1)$ to its parabolic subgroups.

This article generalizes Duflo's conjecture in the framework of Kirillov's conjecture in spirit by Liu-Yu \cite{LY}. We obtain a generalization for the irreducible unitary representations of $\mathrm{GL}_n(k)$ (for $k=\mathbb{R}$ or $\mathbb{C}$) which are attached to some coadjoint orbits in section \ref{sec-5}.

Let's say a few words about the Kirillov's conjecture. The conjecture says that the restriction of any irreducible unitary representation of $\mathrm{GL}_n(k)$ (for $k=\mathbb{R},\mathbb{C}$ or $Q_p$) to the mirabolic subgroup is also irreducible. It was proved by Bernstein for $p$-adic fields \cite{Ber}, Sahi \cite{S} for tempered representations of $\mathrm{GL}_n(\mathbb{R})$ or $\mathrm{GL}_n(\mathbb{C})$, and
Sahi-Stein \cite{SS} for Speh representations of $\mathrm{GL}_n(\mathbb{R})$, and Baruch
\cite{Bar} for the archimedean fields.

Here are the main results of this article. Let $k=\mathbb{R}$ or $\mathbb{C}$,  let $G=\mathrm{GL}_n(k)$, and let $P=P_n(k)$ be the mirabolic subgroup of $G$, consisting of the elements whose last row is $(0,\cdots,0,1)$. Let $\mathfrak{p}$ denote the Lie algebra of $P$ and $\mathfrak{p}^*$ denote the dual of $\mathfrak{p}$.
Assume that $\pi$ is an irreducible unitary representation of $G$ and $\pi$ is attached to the $G$-coadjoint orbit $\mathcal{O}_{\pi}$ in section \ref{sec-5}.
It turns out that the image of the moment map
$\mathrm{p}: \mathcal{O}_{\pi}\to \mathfrak{p}^*$
contains finite $P$-coadjoint orbits, and there is a unique dense $P$-coadjoint orbit in $\mathrm{p}(\mathcal{O}_{\pi})$, denoted by $\Omega$. The moment map $\mathrm{p}$ is proper over $\Omega$. Moreover, $\pi|_{P}$ is attached to $\Omega$ in the sense of Duflo. It explores more about the geometry of the Kirillov's conjecture.

The article is organized as follows.

In section 2, we present the classification of the $P_n(k)$-coadjoint orbits and its proof, which comes from Liu-Yu \cite{LY}.

In section 3, we compute the moment map and obtain similar results as Liu-Yu \cite{LY} on the geometric side.

In section 4, we summarize some results of Kirillov's conjecture, from which we obtain the restrictions of all irreducible unitary representations of $\mathrm{GL}_n(k)$ to $P_n(k)$.

In section 5, firstly, we summarize the orbit method for reductive groups, and obtain the correspondence between coadjoint orbits and irreducible unitary representations of $\mathrm{GL}_n(k)$. Secondly, we show how to construct irreducible unitary representations from coadjoint orbits of $P_n(k)$ in the sense of Duflo.

In section 6, our generalization is presented and proved by comparing the results of the moment maps and the restrictions of irreducible unitary representations of $\mathrm{GL}_n(k)$ to $P_n(k)$.

\section{$P_n(k)$-coadjoint orbit}\label{sec-2}

This section is mainly adapted from Liu-Yu \cite{LY}, which will be used in the rest of the article.

\subsection{Coadjoint action, the dual map and the moment map}

Let $k$ be the field $\mathbb{R}$ or $\mathbb{C}$, and let $n\in \mathbb{Z}_{+}$. Set $G_n(k)=\mathrm{GL}_n(k)$ and \[P_n(k)=\{\left(\begin{array}{cc}A&\alpha\\ 0&1\end{array}\right) : A\in \mathrm{GL}_{n-1}(k),\alpha\in k^{n-1}\}\]
be the mirabolic subgroup of $G_n(k)$.  $\mathfrak{g}_n(k)$ denotes the Lie algebra of $G_n(k)$ and  \[\mathfrak{p}_n(k)=\{\left(\begin{array}{cc}A&\alpha\\ 0&0\end{array}\right) : A\in \mathfrak{gl}_{n-1}(k),\alpha\in k^{n-1}\}\]
denotes the Lie algebra of $P_n(k)$. $\mathfrak{g}_n(k)^*$ and $\mathfrak{p}_n(k)^*$ denote the dual spaces of $\mathfrak{g}_n(k)$ and $\mathfrak{p}_n(k)$, respectively.

$G_n(k)$ acts on $\mathfrak{g}_n(k)^*$ through
\[(g\cdot f)\xi=f(g^{-1}\cdot \xi), \forall g\in G_n(k), \forall f\in \mathfrak{g}_n(k)^*,\forall \xi\in \mathrm{g}_n(k).\]
This is called the coadjoint action, and a $G_n(k)$-orbit in $\mathfrak{g}_n(k)^*$ is called a coadjoint orbit. Similarly, one can define the coadjoint action of $P_n(k)$ on $\mathfrak{p}_n(k)^*$ and the corresponding coadjoint orbit.

By taking trace, we can define a $G_n(k)$-invariant nondegenerate bilinear form on $\mathfrak{g}_n(k)\times \mathfrak{g}_n(k)$,
\[(X,Y)\mapsto \mathrm{tr}(XY), X,Y\in \mathfrak{g}_n(k).\]
We have a $G_n(k)$-module isomorphism
\[\mathrm{pr}: \mathfrak{g}_n(k)\to \mathfrak{g}_n(k)^*,  X\mapsto (Y\mapsto \mathrm{tr}(XY), \forall Y\in \mathfrak{g}_n(k)), \forall X\in \mathfrak{g}_n(k).\]

Let $\mathrm{pr}': \mathfrak{g}_n(k)\to \mathfrak{p}_n(k)^*$, $ X\mapsto (Y\mapsto \mathrm{tr}(XY),\forall Y\in \mathfrak{p}_n(k)), \forall X\in \mathfrak{g}_n(k)$, then
\[\ker \mathrm{pr}'=\{\left(\begin{array}{cc} 0&\alpha\\ 0&t\end{array}\right) : \alpha\in k^{n-1}, t\in k \}.\]
Let $\bar{\mathfrak{p}}_n(k)=\{\left(\begin{array}{cc}A&0 \\ \alpha&0\end{array}\right) | A\in \mathfrak{g}_{n-1}(k), \alpha^t\in k^{n-1}\}$, then $\mathfrak{g}_n(k)=\ker \mathrm{pr}'\oplus \bar{\mathfrak{p}}_n(k)$,  and $\mathrm{pr}'|_{\bar{\mathfrak{p}}_n(k)}$ is an isomorphism.

We use $\mathrm{pr}_n$ (resp. $\mathrm{pr}'_n$) to denote $\mathrm{pr}$ (resp. $\mathrm{pr}'$) if there is ambiguity.

\subsection{Classification of $P_n(k)$-coadjoint orbit}

Let
\[L_n(k)=\{\left(\begin{array}{cc} A &0\\ 0& 1\end{array}\right)\ |\ A\in G_{n-1}(k)\} \]
and
\[N_n(k)=\{\left(\begin{array}{cc} I_{n-1} &\alpha \\ 0& 1\end{array}\right)\ |\ \alpha\in k^{n-1} \}.\]
Let $\mathfrak{l}_n(k)$ and $\mathfrak{n}_n(k)$ denote the Lie algebra of $L_n(k)$ and $N_n(k)$, respectively.

Then $L_n(k)$ is the Levi subgroup of $P_n(k)$, $N_n(k)$ is the unipotent radical of $P_n(k)$, and $P_n(k)=L_n(k)\ltimes N_n(k)$. Moreover, we have the following exact sequence as $P_n(k)$-module
\[ 0\to \mathfrak{n}_n(k)\to \mathfrak{p}_n(k)\to \mathfrak{l}_n(k)\to 0,\]
so is
\[ 0\to \mathfrak{l}_n(k)^*\to \mathfrak{p}_n(k)^*\to \mathfrak{n}_n(k)^*\to 0.\]
We will regard $\mathfrak{l}_n(k)^*$ (resp. $\mathfrak{n}_n(k)^*$) as a $P_n(k)$-submodule (resp. quotient module) of $\mathfrak{p}_n(k)^*$ in the following.

Let $h\in \mathfrak{p}_n(k)^*$ such that $h|_{\mathfrak{l}_n(k)}=0$, $h|_{\mathfrak{n}_n(k)}\neq 0$. We have following key proposition.

\begin{Proposition}\label{ob}
\begin{itemize}
\item[(1)]Every element in $\mathfrak{p}_n(k)^*$ is $P_n(k)$-conjugated to an element in $\mathfrak{l}_n(k)^*$ or $h+\mathfrak{l}_n(k)^*$.
\item[(2)]$\mathrm{Stab}_P(h|_{\mathfrak{n}_n(k)})=L_n(k)^h\ltimes N_n(k)$ and $(h+\mathfrak{l}_n(k)^*)/N_n(k)$ is $h+(\mathfrak{l}_n(k)^h)^*$.
\item[(3)]The stabilizer $L_n(k)^h\cong P_{n-1}(k)$.
\end{itemize}
 As a result,
\[\mathfrak{p}_n(k)^*/P_n(k) \cong \mathfrak{l}_n(k)^*/L_n(k) \bigsqcup \mathfrak{p}_{n-1}(k)^*/P_{n-1}(k).\]
\end{Proposition}

\begin{proof}
(1) Take any $x\in \mathfrak{p}_n(k)^*$, if $x|_{\mathfrak{n}_n(k)}=0$, then $x\in \mathfrak{l}_n(k)^*$.
Assume that $x|_{\mathfrak{n}_n(k)}\neq 0$. As $L_n(k)$ acts on $\mathfrak{n}_n(k)^*$ transitively, $x$ is conjugated to an element in $h +\mathfrak{l}_n(k)^*$.

(2) As $N_n(k)$ is abelian, $N_n(k)$ acts trivially on $\mathfrak{n}_n(k)^*$, so
\[N_n(k)\subset \mathrm{Stab}_P(h|_{\mathfrak{n}_n(k)}).\]
Using $P_n(k)=L_n(k)\ltimes N_n(k)$, we only need to compute $\mathrm{Stab}_{L_n(k)}(h|_{\mathfrak{n}_n(k)})$.
As $h|_{\mathfrak{l}_n(k)}=0$, $\mathrm{Stab}_{L_n(k)}(h|_{\mathfrak{n}_n(k)})=\mathrm{Stab}_{L_n(k)}(h)=L_n(k)^h$.

 Let $X \in \mathfrak{n}_n(k)$, $\exp(X)\in N_n(k)=\exp\mathfrak{n}_n(k) $, $\exp(X)$ acts on $h$ by
\[h\mapsto (Y\mapsto h(\mathrm{Ad}(\exp(-X))(Y)), \forall Y\in \mathfrak{p}_n(k)).\]
Since $\mathrm{Ad}(\exp(-X))(Y)=Y-[X,Y], \forall Y\in \mathfrak{p}_n(k)$, we get that
\[h(\mathrm{Ad}(\exp(-X))(Y))=h(Y)+h(-[X,Y])=(h+\mathrm{ad}(X)h)(Y), \forall Y\in \mathfrak{p}_n(k).\]
As the pairing between $\mathfrak{l}_n(k)$ and $\mathfrak{l}_n(k)^*$ is nondegenerate, we only need to check that $\mathfrak{l}_n(k)^h$ is the zero set of $\{\mathrm{ad}(X)(h)|X\in\mathfrak{n}\}$. Take any $l\in \mathfrak{l}_n(k)$ such that $l$ vanishes at  $\{\mathrm{ad}(X)(h)|X\in\mathfrak{n}\}$, that is,
\[l(\mathrm{ad}(X)(h))=h(-[X,l])=-X(\mathrm{ad}(l)(h))=0,\]
so $l\in \mathfrak{l}_n(k)^h$.

(3) It is easy to check directly.
\end{proof}

We go to calculate $\mathrm{Stab}_{P_n(k)}(x)$ for any $x\in \mathfrak{p}_n(k)^*$.

\begin{Proposition}\label{st1}
Let $x$ be an element in $h+\mathfrak{l}_n(k)^*$, and let $[x]$ be the image of $x$ in $h+(\mathfrak{l}_n(k)^h)^*$ as Proposition \ref{ob}.(2) and $x'=[x]-h$. Then $\mathrm{Stab}_{P_n(k)}(x)=\mathrm{Stab}_{L_n(k)^h}(x')$.
\end{Proposition}

\begin{proof}
Take any $g\in \mathrm{Stab}_{P_n(k)}(x)$, then $g=ln$, $l\in L_n(k)$, $n\in N_n(k)$. As $n$ stabilizes $h|_{\mathfrak{n}_n(k)}$ and  $\mathfrak{l}_n(k)^*$, $l$ stabilizes $h|_{\mathfrak{n}_n(k)}$, hence also stabilizes $h$. As a result, we get a group morphism
\[\phi: \mathrm{Stab}_{P_n(k)}(x) \to \mathrm{Stab}_{L_n(k)^h}(x'), g=ln \mapsto l .\]

We claim that $\phi$ is an isomorphism.
If $l \in \mathrm{Stab}_{L_n(k)^h}(x')$, Proposition \ref{ob} shows
\[\mathfrak{l}_n(k)^*/\mathrm{ad}(\mathfrak{n}_n(k))h=(L_n(k)^h)^*,\]
and one can choose an $n_l\in N_n(k)$, such that $ln_l\in \mathrm{Stab}_{P_n(k)}(x) $. So we get a group morphism
\[\psi: \mathrm{Stab}_{L_n(k)^h}(x') \to \mathrm{Stab}_{P_n(k)}(x), l\mapsto ln_l.\]

It is obvious that $\phi\circ\psi=1$, and it remains to check that $\phi$ is injective. Let $ln\in \ker(\phi)$, then $n$ stabilizes $h$. It can be checked directly that $n=1$.
\end{proof}

\begin{Proposition}\label{st2}
Let $x$ be an element in $(\mathfrak{l}_n(k))^*$, then $\mathrm{Stab}_{P_n(k)}(x)=\mathrm{Stab}_{L_n(k)}(x)\ltimes N_n(k)$.
\end{Proposition}

\begin{proof}
It is easy to check directly.
\end{proof}

In conclusion, we get the following result.

\begin{Theorem}\label{obs}
Every element of $\mathfrak{p}_n(k)^*$ is $P_n(k)$-conjugated to an element of the form $\alpha=\mathrm{pr}'(\xi)$, where
$\xi=\left(\begin{array}{cc}A&\\ &J_j\end{array}\right),$
where $A\in \mathfrak{g}_{n-j}(k)$, and
\[J_j=\left(\begin{array}{cccc}  0&&& \\ 1&\ddots&& \\ &\ddots &\ddots& \\ &&1&0 \end{array}\right)\in \mathfrak{g}_j(k).\]
The stabilizer $\mathrm{Stab}_{P_n(k)}(\alpha)$ is isomorphic to $\mathrm{Stab}_{G_{n-j}(k)}(A)\ltimes N_{n-j+1}(k)$.
\end{Theorem}

\begin{proof}
Choose $h=\mathrm{pr}'(\left(\begin{array}{ccc} 0_{(n-1)\times (n-2)}&0_{(n-1)\times 1}& 0_{(n-1)\times 1}\\ 0_{1\times (n-2)}&1&0\end{array}\right))$, applying Proposition \ref{ob}, we get that every element of $\mathfrak{p}_n(k)^*$ is $P_n(k)$-conjugated to $\mathrm{pr}'(\left(\begin{array}{cc}A&\\ &0\end{array}\right))$ for some $A\in \mathfrak{g}_{n-1}(k)$ or $\mathrm{pr}'(\left(\begin{array}{cc}A&\\ &0\end{array}\right))+h$, where $A\in \bar{\mathfrak{p}}_{n-1}(k)$. Using the result of Proposition \ref{ob} and by induction, one can prove the first statement. Applying Proposition \ref{st1} and \ref{st2}, we can get the stabilizer as the statement.
\end{proof}

\begin{Definition}\label{depth}
Define the depth of such element in $\mathfrak{p}_n(k)^*$ as $j$. It is well defined by the proof of Theorem \ref{obs}.
\end{Definition}

\begin{Corollary}
The elements with depth $n$ form a unique dense open $P_n(k)$-orbit in $\mathfrak{p}_n(k)^*$ and they are all strongly regular elements in $\mathfrak{p}_n(k)^*$.
\end{Corollary}

\begin{proof}
Applying Theorem \ref{obs}, we have $\alpha\in \mathfrak{p}_n(k)^*$ has depth $n$ if and only if $\mathrm{Stab}_{P_n(k)}(\alpha)$ has minimal dimension, if and only if $\mathrm{Stab}_{P_n(k)}(\alpha)=1$. So we get the statement immediately.
\end{proof}

\begin{Example}
Set
\[B= \left(\begin{array}{cccc}  a_1&&& \\ &\ddots&& \\ &&a_{n-1}& \\ b_1 & \cdots & b_{n-1}&0 \end{array}\right)\in \mathfrak{g}_n(k),\]
$a_i\neq a_j$ when $1\leq i\neq j\leq n-1$, $b_i\neq 0,$ when $1\leq i\leq n-1$, and $f=\mathrm{pr}'(B)\in \mathfrak{p}_n(k)^*$. Then $f$ has depth $n$.
\end{Example}

\begin{proof}
It is sufficient to check that $\mathrm{Stab}_{P_n(k)}(f)=\{1\}.$ See Liu-Yu \cite{LY} for a proof. An alternative proof is given in Lemma \ref{st} and Lemma \ref{nst}.
\end{proof}

\section{Geometry of the moment map $\mathrm{p}:\mathcal{O}_f\to \mathfrak{p}^*$}\label{sec-3}

\subsection{Moment map in the $\mathrm{GL}_n(\mathbb{C})$ case }

Let $n\in \mathbb{Z}_+$, let $G=G_n(\mathbb{C})$ and let $P=P_n(\mathbb{C})$ be the mirabolic subgroup of $G$. $\mathfrak{g}$ (resp. $\mathfrak{p}$) denotes the Lie algebra of $G$ (resp. of $P$).

Let $\mathfrak{g}^*$ (resp. $\mathfrak{p}^*$) denote the dual space of $\mathfrak{g}$ (resp. of $\mathfrak{p}$).
By the Lie theory, we have
\[n=\min_{x\in \mathfrak{g}^*} \dim (\mathrm{Stab}_G(x)).\]
For any element $f\in \mathfrak{g}^*$, we call $f$  a regular element in $\mathfrak{g}^*$, or say that $f$ is regular, if $\dim(\mathrm{Stab}_G(f))=n$. And we call a $G$-coadjoint orbit $\mathcal{O}_f$  a regular orbit if $f$ is regular.

We will calculate the moment map $\mathrm{p}: \mathcal{O}_f \to \mathfrak{p}^*$ for any $f\in \mathfrak{g}^*$.
We find that the result about any regular element $f\in\mathfrak{g}^*$ is similar to the results about semisimple regular elements, for example, $\mathrm{p}(\mathcal{O}_f)$ contains the unique strongly regular orbit of $\mathfrak{p}^*$ (see Liu-Yu \cite{LY}). Using the results about the regular elements, we can settle the cases of singular elements in $\mathfrak{g}^*$.

Furthermore, using the results of the case of $\mathfrak{g}_n(\mathbb{C})$, one can settle the case of $\mathfrak{g}_n(\mathbb{R})$ immediately.

\subsubsection{Regular orbits}

Write
\[J_s(a)=\left(\begin{array}{cccc} a& &&\\ 1&a &&\\ &\ddots &\ddots &\\ && 1& a\end{array}\right)\in \mathfrak{g}_s(\mathbb{C}), s\in \mathbb{Z}\ \text{and}\ s\geq 2, a\in \mathbb{C}, \]
and  $J_1(a)=a_{1\times 1}, a\in \mathbb{C}$. Set $J_s=J_s(0)$, $\forall s\in \mathbb{Z}_+$.

Let $a_i\in \mathbb{C}$, $s_i\in \mathbb{Z}_+$, $\forall 1\leq i\leq m$, such that $a_i\neq a_j$ when $1\leq i\neq j\leq m$, and $\sum_{i=1}^ms_i=n$. Set
\[\xi=\mathrm{diag}(J_{s_1}(a_1),\cdots,J_{s_m}(a_m))\in \mathfrak{g},\]
and set $f=\mathrm{pr}(\xi)\in \mathfrak{g}^*$. It is clear that every regular $G$-coadjoint orbit in $\mathfrak{g}^*$ is of the form $\mathcal{O}_f$.

Firstly, we describe all $P$-orbits in $\mathcal{O}_f$.

For any $s\in \mathbb{Z}_+$, let
\[L_s=\mathrm{Stab}_{G_s(\mathbb{C})}(J_s)=\{\left(\begin{array}{cccc} x_1&&&\\ x_2& x_1 &&\\ \vdots& \ddots & \ddots &\\ \ x_s &\cdots &x_2&x_1\end{array}\right) |\  x_i\in \mathbb{C},1\leq i\leq s,  x_1\neq 0\},\]
Set
\[L=\mathrm{Stab}_G(f)=\{\mathrm{diag}(X_{s_1},\cdots,X_{s_m})\ |\ X_{s_i}\in L_{s_i}, 1\leq i\leq m\}.\]
We have $\mathcal{O}_f\cong G/L$ and $P\setminus \mathcal{O}_f \cong P\setminus G/L$.

Define the right action of $G $ on $\mathbb{C}^n$ by
\[v\cdot g=(x_1,\cdots,x_n) A, \]
where $v=(x_1,\cdots,x_n)^t\in \mathbb{C}^n$ and $g=A\in G_n(\mathbb{C})$. Set $v_0=(0,\cdots,0,1)^t$, we get $\mathrm{Stab}_G(v_0)=P$, and $\mathbb{C}^n- \{0\} \cong P\setminus G$.

To find the representatives of $P\setminus \mathcal{O}_f$, we need to find the representatives of $(\mathbb{C}^n- \{0\})/L$.

\begin{Proposition}
Set $I_0=\{1,2,\cdots,n\}, K_0=\{1,\cdots,m\}$ and $s_0=0$. For any $\emptyset\neq K\subseteq K_0$ and $r_k\in \{1,2,\cdots,s_k\}$ for $k\in K$, define
\[I=\{\sum_{j=0}^{k-1}s_j+r_k | \ k\in K\}\subseteq I_0\]
attached to $K$ and $\{r_k, k\in K\}$.
Define $v_I=(x_1,\cdots,x_n)$ by $x_i=1$ when
$i\in I$ and $x_i=0$ when $i\notin I$. Let $I_a$ be the set of all $I$'s  constructed as above.

Then $\{v_I | \ I\in I_a\}$ form all different representatives of $(\mathbb{C}^n- \{0\})/L$.
\end{Proposition}
\begin{proof}
The $L$-orbit of $v_I$ is the set
\begin{align*}\{(y_1,\cdots,y_n)^t\ &|\ y_i\in\mathbb{C}, 1\leq i\leq n; y_i\neq 0\ \text{if}\ i\in I; y_i=0\ \text{if}\ i\notin J, \\& J=\bigsqcup_{k\in K}\{\sum_{j=0}^{k-1}s_j+1,\cdots,\sum_{j=0}^{k-1}s_j+r_k\}\}.\end{align*}
It follows that $\{v_I |\ I\in I_a\}$ are all different representatives of $(\mathbb{C}^n- \{0\})/L$.
\end{proof}

Let $g_I\in G$ be the element corresponding to the representative $v_I$ under the isomorphism $\mathbb{C}^n- \{0\} \cong P\setminus G$, then
\[G=\bigsqcup_{I\in I_a}Pg_IL.\]
So there are $\# I_a=\prod_{1\leq i\leq m}(s_i+1)-1$ $P$-orbits in $\mathcal{O}_f$,
\[\{P\cdot (g_I\cdot f)\ |\ I\in I_a\}.\]

Set $S=\{s_1,s_1+s_2,\cdots,n\}$, then $Pg_{S}L$ is the unique dense open subset of $G$ among $\{Pg_IL,I\in I_a\}$, since $v_S\cdot L$ is the unique dense $L$-orbit in $\mathbb{C}^n-\{0\}$. As a consequence, $P\cdot(g_S\cdot f)$ is the unique dense open $P$-orbit in $\mathcal{O}_f$.

By definition, we can choose $\{g_I | \ I\in I_a\}$  as follows.

When $n\in I$, $g_I=\left(\begin{array}{cc} I_{n-1}& 0\\ \beta& 1\end{array}\right)$, $\beta=(x_1,\cdots,x_{n-1})$ is defined by $x_i=1$ when $i\in I-\{n\}$ and $x_i=0$ when $i\notin I-\{n\}$.

When $n\notin I$, let $k=\max\{i: i\in I\}$, $g_I=\left(\begin{array}{ccc} I_{k-1}&&\\ && I_{n-k}\\ \beta& 1 & \end{array}\right)$, $\beta=(x_1,\cdots,x_{k-1})$ is defined by $x_i=1$ when $i\in I-\{k\}$ and $x_i=0$ when $i\notin I-\{k\}$.

We go to calculate the image of each $P$-orbit of $g_I\cdot f$ under the moment map $\mathrm{p}: \mathcal{O}_f\to \mathfrak{p}^*$ and find that $\mathrm{p}(\mathcal{O}_f)$ contains the unique dense open $P$-orbit in $\mathfrak{p}^*$, more precisely, $\mathrm{p}(g_S\cdot f)$ is a strongly regular element in $\mathfrak{p}^*$.

\begin{Lemma}\label{st}
We have $\mathrm{Stab}_P(g_S\cdot f)=1$, so the dimension of $P$-orbit of $g_S\cdot f$ equals to $\mathrm{dim} P=n^2-n$.
\end{Lemma}

\begin{proof}
It is obvious that
\[\mathrm{Stab}_P(g_S\cdot f)=P\cap g_S\mathrm{Stab}_G(f)g_S^{-1}.\]
And
\[P\cap g_S\mathrm{Stab}_G(f)g_S^{-1}=P\cap g_SLg_S^{-1}=g_S(g_S^{-1}Pg_S\cap L)g_S^{-1}.\]
Since $g_S^{-1}Pg_S$ is the stabilizer of $v_0\cdot g_S=v_S$,  we have $g_S^{-1}Pg_S\cap L=\mathrm{Stab}_L(v_S)$.

By direct calculation, $\mathrm{Stab}_L(v_S)=1$, so $\mathrm{Stab}_P(g_S\cdot f)=1$.
\end{proof}

For convenience, let $[U_1,\cdots,U_k]$ denote an $n\times n$ matrix with $k$ blocks, where the $i$-th block is an $n\times n_i$ ($1\leq n_i\leq n$) submatrix $U_i$ for $1\leq i\leq k$, and
$\sum_{i=1}^k n_i=n$.

\begin{Lemma}\label{nst}
If $x\in \mathfrak{p}^*$ is not strongly regular, then for any element $z\in \mathrm{p}'^{-1}(x)$, the $P$-orbit of $z$ has dimension $< n^2-n$. Here $\mathrm{p}': \mathfrak{g}^*\to \mathfrak{p}^*$ is the natural projection.
\end{Lemma}

\begin{proof}
Take any element $x=\mathrm{pr}'([X,0_{n\times 1}])\in \mathfrak{p}^*$ which is not strongly regular,
by the definition, $\mathrm{p}'^{-1}(x)=\{\mathrm{pr}([X,\tilde{Y}]), \tilde{Y}\in \mathbb{C}^n$\}. Fix an element $z\in \mathrm{p}'^{-1}(x)$, $z=\mathrm{pr}([X,Y])$ for some $Y=(y_1,\cdots,y_n)^t\in \mathbb{C}^n$.

Since the $P$-orbit of $z$ has dimension
\[\mathrm{dim} P-\mathrm{dim} (\mathrm{Stab}_P(z))=n^2-n-\mathrm{dim}(\mathrm{Stab}_P(z)),\]
we only need to prove $\mathrm{dim} (\mathrm{Stab}_P(z))\geq 1$.

Since $\{\mathrm{pr}([0,\tilde{Y}]), \tilde{Y}\in \mathbb{C}^n\}$ is a $P$-stable subspace, $\mathrm{Stab}_P(z)\subseteq \mathrm{Stab}_P(x)$.

Using the result of the classification of $P$-coadjoint orbits in $\mathfrak{p}^*$, we know that $x$ is $P$-conjugated to an element of the form $\alpha=\mathrm{pr}'(\left(\begin{array}{cc}A&\\ &J_k\end{array}\right))$, where $A\in \mathfrak{g}_{n-k}(\mathbb{C})$,
\[\mathrm{Stab}_P(\alpha)\cong \mathrm{Stab}_{G_{n-k}(\mathbb{C})}(A)\ltimes N_{n-k+1}(\mathbb{C})\]
with dimension $\geq 2(n-k)$ and $k\leq n-1$ as $x$ is not strongly regular.
It doesn't matter to replace $x$ by the elements in its $P$-orbit and we assume that $x=\alpha$.

By direct calculation, we get
\begin{align*}\mathrm{Stab}_P(x)&=\{\left(\begin{array}{cc} C &\\&I_{k}\end{array}\right) | \ C\in G_{n-k}(\mathbb{C}), CAC^{-1}=A\}\ltimes \{\left(\begin{array}{cc} I_{n-k} &N\\&I_{k}\end{array}\right)\\
& | \ N=[N_1,\cdots,N_k],\ N_i\in \mathbb{C}^{n-k}, 1\leq i\leq k, N_i=A^{i-1}N_1, 2\leq i\leq k\}.
\end{align*}

Let $t=\left(\begin{array}{cc} C &CN\\&I_{k}\end{array}\right)\in \mathrm{Stab}_P(x)$, then $t\in \mathrm{Stab}_P(z)$ if and only if
\[\left(\begin{array}{cc} C &CN\\&I_{k}\end{array}\right)(\left(\begin{array}{cc}A&\\ &J_k\end{array}\right)+[0,Y])\left(\begin{array}{cc} C^{-1} &-N\\&I_{k}\end{array}\right)=\left(\begin{array}{cc}A&\\ &J_k\end{array}\right)+[0,Y], \]
So
\[ t\in \mathrm{Stab}_P(z) \Leftrightarrow  -CA^kN_1+C\left(\begin{array}{c}y_1\\ \vdots \\y_{n-k}\end{array}\right)+CN\left(\begin{array}{c}y_{n-k+1}\\ \vdots \\ y_n\end{array}\right)=\left(\begin{array}{c}y_1\\ \vdots \\y_{n-k}\end{array}\right).\]

Since $1\in \mathrm{Stab}_P(z)$, we obtain a nonempty subvariety $\mathrm{Stab}_P(z)$ of $\mathrm{Stab}_P(x)$ given by $n-k$ algebraic equations. Therefore
\[\dim(\mathrm{Stab}_P(z)) \geq \mathrm{dim}(\mathrm{Stab}_P(x))-(n-k).\]
Since $\mathrm{dim}(\mathrm{Stab}_P(x))\geq 2(n-k)$ and $n-k\geq 1$, we get $\dim(\mathrm{Stab}_P(z))\geq 1$.
\end{proof}

\begin{Proposition}\label{gS}
$\mathrm{p}(g_S\cdot f)$ is a strongly regular element in $\mathfrak{p}^*$.
\end{Proposition}

\begin{proof}
By Lemma \ref{st},  the $P$-orbit of $g_S\cdot f$ has dimension $n^2-n$. Applying Lemma \ref{nst}, we see that $\mathrm{p}(g_S\cdot f)$ is strongly regular.
\end{proof}

For convenience, let $A(I_1,I_2)$ denote the submatrix of the matrix $A$ with rows indexed by $I_1$ and columns indexed by $I_2$, and let $a(I_1)$ denote the subvector of the vector $a$ with index $I_1$.

With the result above, we can calculate the $P$-orbit of $\mathrm{p}(g_I\cdot f)$ for any $I\in I_a$. We find out that  $\{\mathrm{p}(g_I\cdot f)\ | \ I\in I_a\}$ represent all different $P$-orbits in $\mathrm{p}(\mathcal{O}_f)$. More precisely, we have the following result.

\begin{Proposition}\label{rg}
For any $\emptyset\neq K\subseteq K_0$ and $r_k\in \{1,2,\cdots,s_k\}$ for $k\in K_0$, define $I=\{\sum_{j=0}^{k-1}s_j+r_k | \ k\in K\}$ and $g_I$ as above. Then $\mathrm{p}(g_I\cdot f)$ has depth $d=\sum_{k\in K}r_k$. Moreover, if we set
\[t_k=\left\{\begin{array}{ll} r_k, &k\in K, \\ 0, &k\in K_0-K,\end{array}\right.\]
then $\mathrm{p}(g_I\cdot f)$ is $P$-conjugated to $\mathrm{pr}'(\eta)$, where
\[ \eta= \mathrm{diag}(J_{s_1-t_1}(a_1), \cdots,J_{s_m-t_m}(a_m), J_d).\]
\end{Proposition}

\begin{proof}
Firstly, we change $\mathrm{p}(g_I\cdot f)=\mathrm{pr}'(g_I\cdot \xi)$ to a good form by a series of explicit $P$-conjugations, then we use the result of Proposition \ref{gS} to get the statement.
Set $A_I=g_I\cdot \xi$.

In the case of $n\in I$,  by direct calculation, we have
\[ A_I =\mathrm{diag}(J_{s_1}(a_1),\cdots,J_{s_m}(a_m))+\left(\begin{array}{c}0_{(n-1)\times n}\\ \alpha\end{array}\right),\]
 where $\alpha=(\alpha_1,\cdots,\alpha_m)$, $\alpha_k=(\alpha_{k,1},\cdots, \alpha_{k,s_k})$ is defined by
 \[\alpha_k=\left\{ \begin{array}{ll} (\underbrace{0,\cdots,0}_{s_k}), & k\notin K-\{n\},\\ (a_k-a_m,\underbrace{0,\cdots,0}_{s_k-1}), &k\in K-\{n\}, r_k=1\\ (\underbrace{0,\cdots, 0}_{r_k-2}, 1,a_k-a_m,\underbrace{0,\cdots,0}_{s_k-r_k}), &k\in K-\{n\}, r_k>1.\end{array}\right.\]
We claim that $\mathrm{pr}'(A_I)$ is $P$-conjugated to $\mathrm{pr}'(\zeta)$, where
\[\zeta= \left(\setlength\arraycolsep{0mm}\begin{array}{ccccc}J_{t_1}(a_1)&&&&\\
&J_{s_1-t_1}(a_1)&&&\\
&&\ddots&&\\
&&&J_{t_m}(a_m)&\\
&&&&J_{s_m-t_m}(a_m)\end{array}\right) +\left(\begin{array}{c}0_{(n-1)\times n}\\ \alpha\end{array}\right).  \]

It is sufficient to prove that
$\mathrm{pr}'(A_I)$ is $P$-conjugated to $\mathrm{pr}'(\zeta_k)$ for $0\leq k\leq m$, where
\[\zeta_k= \left(\setlength\arraycolsep{0mm}\begin{array}{cccccccc}J_{t_1}(a_1)&&&&&&&\\
&J_{s_1-t_1}(a_1)&&&&&&\\
&&\ddots&&&&&\\
&&&J_{t_k}(a_k)&&&&\\
&&&&J_{s_k-t_k}(a_k)&&&\\
&&&&&J_{s_{k+1}}(a_{k+1})&&\\
&&&&&&\ddots&\\
&&&&&&&J_{s_m}(a_m)\end{array}\right) +\left(\begin{array}{c}0_{(n-1)\times n}\\ \alpha\end{array}\right),  \]
since $\zeta_m=\zeta$.

Let's prove it by induction on $k$.
When $k=0$, it is trivially true.
Assume that $k\geq 1$ and it is true for $k-1$.

If $k\notin K$, or if $k\in K$ and $r_k=s_k$, then $\zeta_k=\zeta_{k-1}$ and so it is true for $k$. Assume that $k\in K, r_k<s_k$, we go to prove that it is true for $k$. Since we assume that $\mathrm{pr}'(A_I)$ is $P$-conjugated to $\mathrm{pr}'(\zeta_{k-1})$, it is sufficient to prove that $\mathrm{pr}'(\zeta_{k-1})$ is $P$-conjugated to $\mathrm{pr}'(\zeta_k)$.

Step 1. We have
\[\zeta_{k-1}=\left(\begin{array}{ccccc}
B_1&0&0&0&0\\
0&C_1&0&0&0\\
0&D&C_2&0&0\\
0&0&0&B_2&0\\
v_1&w_1&0&v_2&a_m \end{array}\right),\]
where
\[B_1=\mathrm{diag}(J_{t_1}(a_1),J_{s_1-t_1}(a_1),\cdots,J_{t_{k-1}}(a_{k-1}),J_{s_{k-1}-t_{k-1}}(a_{k-1})),\]
\[B_2=\mathrm{diag}(J_{s_{k+1}}(a_{k+1}),\cdots,J_{s_{m-1}}(a_{m-1}),J_{s_m-1}(a_m)),\]
\[C_1=J_{r_k}(a_k), C_2=J_{s_k-r_k}(a_k),v_1=(\alpha_1,\cdots,\alpha_{k-1}), w_1=(\alpha_{k,1},\cdots,\alpha_{k,r_k}),\]
\[ v_2=\left\{\begin{array}{ll}(\alpha_{k+1},\cdots,\alpha_{m-1}), & s_m=1,\\
(\alpha_{k+1},\cdots,\alpha_{m-1},\underbrace{0,\cdots,0}_{s_m-2},1), & s_m>1, \end{array}\right.\]
\[D\ \text{is the}\ (s_k-r_k)\times r_k\ \text{matrix with}\ 1\ \text{in the position}\ (1,r_k)\ \text{and zero elsewhere}.\]

Let
\[P_1=\left(\begin{array}{ccccc}
I_{s_1+\cdots+s_{k-1}}&0&0&0&0\\
0&I_{r_k}&0&0&0\\
0&N_1&I_{s_k-r_k}&0&N_2\\
0&0&0&I_{s_{k+1}+\cdots+s_m-1}&0\\
0&0&0&0&1 \end{array}\right)\]
with $N_1=[0_{(s_k-r_k)\times (r_k-1)},n_1]$, $N_2=-n_1$ and
\[n_1 =(\frac{1}{a_k-a_m},-\frac{1}{(a_k-a_m)^2},\cdots,\frac{(-1)^{s_k-r_k+1}}{(a_k-a_m)^{s_k-r_k}})^t\in \mathbb{C}^{s_k-r_k}.\]

By direct calculation, we get
\[P_1\zeta_{k-1}P_1^{-1}=\left(\begin{array}{ccccc}
B_1&0&0&0&0\\
0&C_1&0&0&0\\
N_2v_1&0&C_2&N_2v_2&a_mN_2-A_2N_2\\
0&0&0&B_2&0\\
v_1&w_1&0&v_2&a_m \end{array}\right).\]

Step 2. Let
\[P_2=\left(\begin{array}{ccccc}
I_{s_1+\cdots+s_{k-1}}&0&0&0&0\\
0&I_{r_k}&0&0&0\\
M_1&0&I_{s_k-r_k}&M_2&0\\
0&0&0&I_{s_{k+1}+\cdots+s_m-1}&0\\
0&0&0&0&1 \end{array}\right),\]
$M_1$ is an $(s_k-r_k)\times (s_1+\cdots+s_{k-1})$ matrix, and $M_2$ is an $(s_k-r_k)\times (s_{k+1}+\cdots+s_m-1)$ matrix,
then $P_2(P_1\zeta_{k-1}P_1^{-1})P_2^{-1}=$
\[\left(\begin{array}{ccccc}
B_1&0&0&0&0\\
0&C_1&0&0&0\\
N_2v_1+(M_1B_1-C_2M_1)&0&C_2&N_2v_2+(M_2B_2-C_2M_2)&a_mN_2-A_2N_2\\
0&0&0&B_2&0\\
v_1&w_1&0&v_2&a_m \end{array}\right).\]
By solving the linear equations, one see that there exists $M_1$ and $M_2$, such that
\[N_2v_1+(M_1B_1-C_2M_1)=0, \text{and}\ N_2v_2+(M_2B_2-C_2M_2)=0.\]
Actually, we can write $B_1=D'+N'$ and $C_2=D''+N''$ as Jordan decomposition, where $D',D''$ are semisimple matrices and $N',N''$ are nilpotent matrices. It is easy to see that the linear map
\[M_1\mapsto M_1N'-N''M_1\]
is nilpotent, the linear map
\[M_1\mapsto M_1D'-D''M_1\]
is semisimple, and they commute with each other.
Moreover, the linear map $M_1\mapsto M_1D'-D''M_1$ is nondegenerate since $B_1$ and $C_2$ have no common eigenvalues. Therefore, we can easily see that the equation $N_2v_1+(M_1B_1-C_2M_1)=0$ has a (unique) solution $M_1$. Similarly we can get $M_2$.

Using such $P_2$ defined as above, we have $\mathrm{pr}'(P_2(P_1\zeta_{k-1}P_1^{-1})P_2^{-1})=\mathrm{pr}'(\zeta_k)$.
By induction, one can finish the proof of the claim.

Let $\zeta'$ be the submatrix of $\zeta$ with the rows and columns indexed by
\[\bigsqcup_{h\in K}\{\sum_{j=0}^{h-1}s_j+1,\cdots,\sum_{j=0}^{h-1}s_j+r_h\}.\]
Applying Proposition \ref{gS}, we see that $\mathrm{pr}'_d(\zeta')$ is $P_d(\mathbb{C})$-conjugated to $\mathrm{pr}'_d(J_d)$, where $\mathrm{pr}'_d: \mathfrak{g}_d(\mathbb{C})\to \mathfrak{p}_d(\mathbb{C})^*$. So $\mathrm{pr}'(\zeta)$ is $P$-conjugated to the element $\mathrm{pr}'(\eta)$ in the statement,
and so is $\mathrm{p}(g_I\cdot\xi)$ .

In the case of $n\notin I, m\in K$,  one can prove the statement similarly. We have
\[A_I=\left(\begin{array}{ccc}
B_1&0&0\\
C&B_2&v_1\\
v_2&0&a_m\end{array}\right),\]
where
\[B_1=\mathrm{diag}(J_{s_1}(a_1),\cdots,J_{s_{m-1}}(a_{m-1}),J_{r_m-1}(a_m)),B_2=J_{s_m-r_m}(a_m),\]
\[C=\left(\begin{array}{c}-v_I({1,\cdots,n-(s_m-r_m)-1}) \\ 0_{(s_m-r_m)\times(n-(s_m-r_m)-1) }\end{array}\right),\]
\[v_1=(1,\underbrace{0,\cdots,0}_{s_m-r_m-1})^t, v_2=\left\{\begin{array}{ll}(\alpha({1,\cdots,n-s_m}),\underbrace{0,\cdots,0}_{r_m-2},1), & r_m\geq 2\\
(\alpha({1,\cdots,n-s_m}), & r_m=1.\end{array}\right. \]
We can use the method in the step 2 to show that $\mathrm{p}(A_I)$ is $P$-conjugated to $\mathrm{pr}(\zeta')$, where \[\zeta'=\left(\begin{array}{ccc}
B_1&0&0\\
0&B_2&v_1\\
v_2&0&a_m\end{array}\right).\]
Then the submatrix of $\zeta'$ with rows and columns indexed by $\{1,\cdots,n\}-\{n-(s_k-r_k),\cdots,n-1\}$ is
\[\left(\begin{array}{cc}
B_1&0\\
v_2&a_m\end{array}\right),\]
which is of the form in the case of $n\in I$. Applying the result of the case of $n\in I$, we can get the statement in the case of $n\notin I, m\in K$.

In the case of $n\notin I, m\notin K$, assume $m_0=\max \{k\in K\}$. By directly calculation, we get that the submatrix of $A_I$ with rows and columns indexed by $\{1,2,\cdots,s_1+\cdots+s_{m_0}-1\}\cup\{n\}$ is of the form in the above two cases. Applying the results of the above two cases, we can get the statement in the case of $n\notin I, m\notin K$.

This completes the proof of the statement.
\end{proof}

\subsubsection{General orbits}

We go to calculate the moment map $\mathrm{p}: \mathcal{O}_f\to \mathfrak{p}^*$ for any $f\in \mathfrak{g}^*$.
Write
\[J^l_k(a)=\mathrm{diag}(\underbrace{J_k(a),\cdots,J_k(a)}_{l\ \text{blocks}}),\]
for any $a\in \mathbb{C}$, $k,l\in \mathbb{Z}_+$.

Let $m\in \mathbb{Z}_+$, $m\leq n$.
Let $a_1$, $\cdots$, $a_m$ $\in \mathbb{C}, a_i\neq a_j$ when $1\leq i\neq j\leq m$.
Set
\[J_{k,a_i}^l=J_k^l(a_i)\]
for $1\leq i\leq m$, $k,l\in \mathbb{Z}_+$.
Let
\[\xi=\mathrm{diag}(J_{k_{11},a_1}^{l_{11}},J_{k_{12},a_1}^{l_{12}},\cdots,J_{k_{1r_1},a_1}^{l_{1r_1}}, \cdots, J_{k_{m1},a_m}^{l_{m1}},J_{k_{m2},a_m}^{l_{m2}},\cdots,J_{k_{mr_m},a_m}^{l_{mr_m}}),\]
where
$r_j\in \mathbb{Z}_+, 1\leq j\leq m$; and $l_{ji}, k_{ji}\in \mathbb{Z}_+$, $\forall 1\leq j\leq m, 1\leq i\leq r_j$;
$k_{ji_1}< k_{ji_2}$ when $1\leq j\leq m$, $1\leq i_1<i_2\leq r_j$; and
\[\sum_{1\leq i\leq m, 1\leq i\leq r_j}k_{ji}\cdot l_{ji}=n.\]
So $\xi$ has $m$ different eigenvalues, denoted by $a_j, 1\leq j\leq m$, and has $r_j$ different kinds of Jordan blocks for the eigenvalue $a_j$, with size $k_{ji}$ and $l_{ji}$ pieces of such size for $1\leq i\leq r_j$.

Set $f=\mathrm{pr}(\xi)$. It is clear that every orbit in $\mathfrak{g}^*$ is of the form $\mathcal{O}_f$.

Now, we fix a $\xi$, and set $\xi_j=\mathrm{diag}(J_{k_{j1},a_j}^{l_{j1}},J_{k_{j2},a_j}^{l_{j2}},\cdots,J_{k_{jr_j},a_j}^{l_{jr_j}})$, and $n_j=\sum_{r=1}^{r_j}k_{jr}\cdot l_{jr}$ for $1\leq j\leq m$. For convenience, set $n_0=0$, $k_{j0}=0, l_{j0}=0$ for $1\leq j\leq m$.

Let $M_{s\times s}(\mathbb{C})$ be the set of all $s\times s$ matrices over $\mathbb{C}$. Write
\[K(x_1,\cdots,x_s)=\left(\begin{array}{cccc} x_1&&&\\ x_2& x_1 &&\\ \vdots& \ddots & \ddots &\\ \ x_s &\cdots &x_2&x_1\end{array}\right), x_i\in \mathbb{C}, 1\leq i\leq s,\]
then
\[\{X\in M_{s\times s}(\mathbb{C})\ |\ XJ_s=J_sX\}=\{K(x_1,\cdots,x_s)\ |\  x_i\in \mathbb{C}, 1\leq i\leq s\},\]
and we have the following Proposition.

\begin{Proposition}\label{LK1} For any $a\in \mathbb{C}$, $k,l\in \mathbb{Z}_+$,
\begin{align*}L_{k,l}&=\mathrm{Stab}_{\mathrm{GL}_{k\cdot l}(\mathbb{C})}(J^l_k(a))=
\{\left(\begin{array}{ccc}
K(x_{11,1},\cdots,x_{11,k})&\cdots &K(x_{1l,1},\cdots,x_{1l,k})\\
\vdots &\ddots &\vdots\\
K(x_{l1,1},\cdots,x_{l1,k})&\cdots &K(x_{ll,1},\cdots,x_{ll,k})
\end{array}\right) \\
&|\ x_{ij,h}\in \mathbb{C}, \forall 1\leq i,j\leq l, 1\leq h\leq k, \det(x_{ij,1})_{1\leq i,j\leq l}\neq 0\}.\end{align*}
\end{Proposition}

\begin{proof}
By direct calculation, we see that the determinate of the element on the right hand side equals to $(\det(x_{ij,1})_{1\leq i,j\leq l})^k$ and get the statement immediately.
\end{proof}

For any $k_1,k_2\in \mathbb{Z}_+$ such that $k_1\leq k_2$, and $x_1,\cdots,x_{k_1}\in \mathbb{C}$,
define $H_1(x_1,\cdots,x_{k_1})$ as the $k_1\times k_2$ matrix
\[[K(x_1,\cdots,x_{k_1}),0_{k_1\times (k_2-k_1)}],\]
and define $H_2(x_1,\cdots,x_{k_1})$ is the $k_2\times k_1$ matrix
\[\left(\begin{array}{c} 0_{(k_2-k_1)\times k_1}\\
 K(x_1,\cdots,x_{k_1})\end{array}\right).\]

\begin{Proposition}
For any $a\in \mathbb{C}$, $k_1,k_2,l_1,l_2 \in \mathbb{Z}_+$, $k_1\leq k_2$,
\begin{align*}M_{k_1,k_2}^{l_1,l_2}&=\{X\in M_{k_1l_1\times k_2l_2}(\mathbb{C})\ |\ X(J_{k_2}^{l_2}(a))=(J_{k_1}^{l_1}(a))X\}\\
&=\{\left(\begin{array}{ccc}H_1(x_{11,1},\cdots,x_{11,k_1})&\cdots &H_1(x_{1l_2,1},\cdots,x_{1l_2,k_1})\\
\vdots &\ddots &\vdots\\
H_1(x_{l_11,1},\cdots,x_{l_11,k_1})&\cdots &H_1(x_{l_1l_2,1},\cdots,x_{l_1l_2,k_1})\end{array}\right)
\\& |\ x_{ij,h}\in \mathbb{C}, \forall 1\leq i\leq l_1,1\leq j\leq l_2, 1\leq h\leq k_1\},\end{align*}
\begin{align*}N_{k_1,k_2}^{l_1,l_2}&=\{X\in M_{k_2l_2\times k_1l_1}(\mathbb{C})\ |\ X(J_{k_1}^{l_1}(a))=(J_{k_2}^{l_2}(a))X\}
\\&=\{\left(\begin{array}{ccc}H_2(x_{11,1},\cdots,x_{11,k_1})&\cdots &H_2(x_{1l_1,1},\cdots,x_{1l_1,k_1})\\
\vdots &\ddots &\vdots\\
H_2(x_{l_21,1},\cdots,x_{l_21,k_1})&\cdots &H_2(x_{l_2l_1,1},\cdots,x_{l_2l_1,k_1})\end{array}\right)
\\& |\ x_{ij,h}\in \mathbb{C}, \forall 1\leq i\leq l_2,1\leq j\leq l_1, 1\leq h\leq k_1\}.\end{align*}
\end{Proposition}

\begin{proof}
It can be checked easily.
\end{proof}

For convenience, we fix $j$ and set $r=r_j$, $k_u=k_{ju}, l_u=l_{ju}$, for $1\leq u\leq r_j$, in the Proposition \ref{Lj} and its proof.
\begin{Proposition}\label{Lj} The stabilizer
\begin{align*}L_{j}&=\mathrm{Stab}_{\mathrm{GL}_{n_j}(\mathbb{C})}(\xi_{j})=
\{\left(\begin{array}{cccc}
X_{11}&X_{12}&\cdots &X_{1r}\\
X_{21}&X_{22}&\cdots &X_{2r}\\
\vdots &\vdots &\ddots &\vdots\\
X_{r1}&X_{r2}&\cdots &X_{rr}
\end{array}\right)
 \ |\ X_{uu}\in L_{k_u,l_u}, \forall 1\leq u\leq r,\\& X_{st}\in M_{k_s,k_t}^{l_s,l_t}, \forall 1\leq s<t\leq r,
X_{st}\in N_{k_s,k_t}^{l_s,l_t}, \forall 1\leq t<s\leq r\}.\end{align*}
\end{Proposition}

\begin{proof}
It is clear that
\begin{align*}\mathrm{Stab}_{\mathrm{GL}_{n_j}(\mathbb{C})}(\xi_{j})&\subseteq
\{\left(\begin{array}{cccc}
X_{11}&X_{12}&\cdots &X_{1r}\\
X_{21}&X_{22}&\cdots &X_{2r}\\
\vdots &\vdots &\ddots &\vdots\\
X_{r1}&X_{r2}&\cdots&X_{rr}
\end{array}\right)\ |\ X_{st}\in M_{k_s,k_t}^{l_s,l_t}, \forall 1\leq s\leq t\leq r,\\
&X_{st}\in N_{k_s,k_t}^{l_s,l_t}, \forall 1\leq t<s\leq r\},\end{align*}
so we only need to prove the element on the right hand side is invertible if and only if $X_{uu}$ is invertible for all $1\leq u\leq r$.
This follows from the fact that the determinant of the element on the right hand side is equal to $\prod_{1\leq u\leq r}\det X_{uu}$.
\end{proof}

It is clear that
\[\mathrm{Stab}_G(f)=\mathrm{Stab}_G(\xi)=\{\mathrm{diag}(X_1,\cdots,X_m)\ |\ X_i\in L_i, 1\leq i\leq m\}.\]
And we have $\mathcal{O}_f\cong G/L$ and $P\setminus \mathcal{O}_f \cong P\setminus G/L$.

Define the right action of $G$ on $\mathbb{C}^n$ by
\[v\cdot g=(x_1,\cdots,x_n) A,\]
where $v=(x_1,\cdots,x_n)^t\in \mathbb{C}^n$ and $g=A\in G_n(\mathbb{C})$. Set $v_0=(0,\cdots,0,1)^t$, we get $\mathrm{Stab}_G(v_0)=P$, and $\mathbb{C}^n-\{0\} \cong P\setminus G$.

To calculate $P\setminus G/L$, we need to calculate $(\mathbb{C}^n-\{0\})/L$. Actually, we only need to calculate $(\mathbb{C}^{n_j}-\{0\})/L_j$. Once we obtain all the different representatives of $(\mathbb{C}^{n_j}-\{0\})/L_j$, denoted by $V_j$, we can get all the different representatives of $(\mathbb{C}^n- \{0\})/L$ as follows,
\[\{v=(v_1,\cdots,v_m)\neq 0\ |\ v_i\in \mathbb{C}^{n_i}, v_i\in V_i \ \text{or}\ v_i=0, 1\leq i\leq m\}.\]
We fix $j$ and adopt the same notation as Proposition \ref{Lj}. We have following proposition about $(\mathbb{C}^{n_j}-\{0\})/L_j$.

\begin{Proposition}\label{db}
Write $R_0=\{1,\cdots,r\}$. For any $\emptyset \neq R\subset R_0$, and $1\leq x_i\leq k_i$ for $i\in R$, such that
\begin{itemize}
\item[(1)] $x_i < x_{i'}$ for any $i,i'\in R$ such that $i<i'$,
\item[(2)] $k_i-x_i < k_{i'}-x_{i'}$ for any $i,i'\in R$ such that $i<i'$,
\end{itemize}
we define
\[I=\{\sum_{u=0}^{i-1}k_ul_u+k_i(l_i-1)+x_i, i\in R\}\subseteq \{1,2,\cdots, n_j\},\]
and $v_I=(y_1,y_2,\cdots,y_{n_j})$, $y_i=1$ when $i\in I$ and $y_i=0$ when $i\notin I$. Let $I_a$ be the set of all $I$'s constructed as above.
Then $\{v_I\ |\ I\in I_a\}$ represent all different $L_j$-orbits in $\mathbb{C}^{n_j}-\{0\}$.
\end{Proposition}

\begin{proof}
We only need to prove that any $L_j$-orbit in $(\mathbb{C}^{n_j}-\{0\})$ is of the form $v_I\cdot L_j$ for some $I$ given as above, and such orbits are different from each other.

For any $k,l\in \mathbb{Z}_+$, set $e(k,l)=(\underbrace{0,\cdots,0}_{k-1},1,\underbrace{0,\cdots,0}_{l-k})\in \mathbb{C}^l$. For $1\leq i\leq r$, set
\[t_i=\{\sum_{u=0}^{i-1}k_ul_u+1,\sum_{u=0}^{i-1}k_ul_u+2,\cdots, \sum_{u=0}^{i}k_ul_u\}.\]
Let $z=(z_1,\cdots,z_{n_j})$ $\in \mathbb{C}^{n_j}-\{0\}$. We use $z(t_i)$ to denote the subvector of $z$ with the positions indexed by $t_i$.

For any $1\leq i\leq r$, we show that one can use the right action of $L_j$ to change $z$ to $z'$, such that $z'(t_i)$ equals $0$ or $e(k_i(l_i-1)+x_i,k_il_i)$, for some $1\leq x_i\leq k_i$, without changing other elements of $z$.

If $z(t_i)=0$, we don't need to change $z$ any more. Otherwise, take $v_i$ to be the maximal integer such that $1\leq v_i\leq k_i$ and
\[z( \sum_{u=0}^{i-1}k_ul_u+k_i\cdot w+v_i)\neq 0\]
for some integer $0\leq w\leq l_i-1$.
Applying Proposition \ref{LK1}, we can realize $z(t_i)$ as the $k_i(l_i-1)+v_i$-th row of some element in $L_{k_i,l_i}$, denoted by $\tilde{L}$. By the right action of the element
\[\mathrm{diag}(I_{k_1l_1},\cdots,I_{k_{i-1}l_{i-1}},\tilde{L}^{-1},I_{k_{i+1}l_{i+1}},\cdots,I_{k_rl_r})\in L_j,\]
one can change $z$ to $z'$, such that $z'(t_i)=e(k_i(l_i-1)+v_i,k_il_i)$, without changing other elements of $z$.

Now we can get all representatives of $(\mathbb{C}^{n_j}-\{0\})/L_j$. For any $\emptyset \neq R\subseteq R_0$, and $1\leq x_i\leq k_i, i\in R$, define
\[J=\{\sum_{u=0}^{i-1}k_ul_u+k_i(l_i-1)+x_i, i\in R\}\subseteq \{1,2,\cdots, n_j\}\]
and $v_J=(y_1,y_2,\cdots,y_{n_j})$, $y_i=1$ when $i\in J$ and $y_i=0$ when $i\notin J$. Then such $v_J$'s include all representatives
of $(\mathbb{C}^{n_j}-\{0\})/L_j$. However, some of them may be the same representatives.

We go to prove that if $x_i\geq x_{i'}$ for some $i<i'$, then $v_{I'}\cdot L_j=v_I\cdot L_j$ with $I'=I-\{\sum_{u=0}^{i'-1}k_ul_u+k_{i'}(l_{i'}-1)+x_{i'}\}$.
For convenience, set  $a=\sum_{u=0}^{i-1}k_ul_u+k_i(l_i-1)$ and $b=\sum_{u=0}^{i'-1}k_ul_u+k_{i'}(l_{i'}-1)$. From Proposition \ref{Lj}, we have
\[\tilde{L}'=I_{n_j}+\sum_{1\leq k\leq  k_{i}-x_i+x_{i'}}E_{a+x_i-x_{i'}+k,b+k} \in L_j.\]
Moreover, $v_{I'}\cdot \tilde{L}'=v_I$, so $v_{I'}\cdot L_j=v_I\cdot L_j$.

If $k_i-x_i\geq k_{i'}-x_{i'}$ for some $i<i'$, then $v_{I''}\cdot L_j=v_I\cdot L_j$, $I''=I-\{\sum_{u=0}^{i-1}k_ul_u+k_i(l_i-1)+x_i\}$.
For convenience, set  $a=\sum_{u=0}^{i-1}k_ul_u+k_i(l_i-1)$ and $b=\sum_{u=0}^{i'-1}k_ul_u+k_{i'}(l_{i'}-1)$. From Proposition \ref{Lj}, we have
\[\tilde{L}''=I_{n_{i'}}+\sum_{1\leq k\leq k_{i'}-x_{i'}+x_i}E_{b+x_{i'}-x_i+k,a+k} \in L_j.\]
Moreover, $v_{I''}\cdot \tilde{L}''=v_I$, so $v_{I''}\cdot L_j=v_I\cdot L_j$.

Now, we have proved that the set of $\{v_I\ |\ I\in I_a\}$ contains all representatives of $(\mathbb{C}^{n_j}-\{0\})/L_j$. It remains to prove that the orbits of $\{v_I\ |\ I\in I_a\}$ are different from each other. Assume $v_{I}\cdot L_j=v_{\tilde{I}}\cdot L_j$, $\tilde{I}$ is corresponding to $\tilde{R}$ and $\tilde{x}_i$. Then by direct calculation of the orbit $v_I\cdot L_j$ and $v_{\tilde{I}}\cdot L_j$, one see that
\[r\in R \Rightarrow r\in \tilde{R},\ \text{and then}\ \tilde{x}_r=x_r.\]
And get that $i\in R$ $\Rightarrow$ $i\in \tilde{R}$, and $\tilde{x}_i=x_i$ for $i=r-1,\cdots,1$, successively.
\end{proof}

\begin{Remark}
Let $R=\{r\}$ and $x_r=k_r$, define $I^{o}=\{n_j\}$, and the corresponding $v_{I^o}$.
From the structure of $L_j$, we can see that the $L_j$-orbit of $v_{I^o}$ is the unique dense open orbit in $\mathbb{C}^{n_j}-\{0\}$.
\end{Remark}

\begin{Remark}
If we just assume that ``$k_{ji}\neq k_{ji'}$ for any $1\leq i\neq i'\leq r_j$" rather than ``$k_{ji}<k_{ji'}$ for any $1\leq i<i'\leq r_j$", the conditions (1) and (2) in the Proposition \ref{db} should be replaced by (1)' $x_i < x_{i'}$ for any $i,i'\in R$ such that $k_{ji}<k_{ji'}$,
and (2)' $k_i-x_i < k_{i'}-x_{i'}$ for any $i,i'\in R$ such that $k_{ji}<k_{ji'}$.
\end{Remark}

With the result above, we can get the representatives of $(\mathbb{C}^n-\{0\})/L$ immediately.

\begin{Proposition}\label{db1}
Set $M_0=\{1,2,\cdots,m\}$ and  $R_{j0}=\{1,2,\cdots,r_j\}$, $1\leq j\leq m$. For any $\emptyset \neq M\subseteq M_0$, and $\emptyset \neq R_j\subseteq R_{j0}$ for each $j\in M$, and $1\leq x_{ji}\leq k_{ji}$ for $i\in R_j$, $j\in M$ satisfying the following two conditions,
\begin{itemize}
\item[(1)] $x_{ji_1} < x_{ji_2}$ for any $i_1,i_2\in R_j$ such that $i_1<i_2$,
\item[(2)] $k_{ji_1}-x_{ji_1} < k_{ji_2}-x_{ji_2}$ for any $i_1,i_2\in R_j$ such that $i_1<i_2$,
\end{itemize}
set
\[I=\bigsqcup_{j\in M}\{\sum_{v=0}^{j-1}n_v+\sum_{u=0}^{i-1}k_{ju}l_{ju}+k_{ji}(l_{ji}-1)+x_{ji}, i\in R_j\}\subseteq \{1,2,\cdots,n\},\]
and $v_I=(y_1,y_2,\cdots,y_n), y_i=1$ when $i\in I$ and $y_i=0$ when $i\notin I$. Let $I_a$ be the set of all such $I$'s.  Then $\{v_I\ |\ I\in I_a\}$ form all different representatives of $(\mathbb{C}^n-\{0\})/L$.
\end{Proposition}

\begin{proof}
It follows from the Proposition \ref{db}.
\end{proof}

\begin{Remark}
Set $M=M_0$, $R_j=\{r_j\}$ and $x_{jr_j}=k_{jr_j}$ for $1\leq j\leq m$. Let $I^{o}$ be the set corresponding to such $M$, $R_j$ and $x_{jr_j}$, i.e. $I^{o}=\{n_1,n_1+n_2,\cdots,n\}$. Then $v_{I^{o}}$ is the unique dense open $L$-orbit in $\mathbb{C}^n-\{0\}.$
\end{Remark}

We can define $g_I$ in the same way as the contents about regular orbits and such $g_I\cdot f$'s form all different representatives of $P\setminus\mathcal{O}_f$ and the orbit of $g_{I^o}\cdot f$ is the unique dense open orbit among them.

We go to calculate the moment map $\mathrm{p}: \mathcal{O}_f\to \mathfrak{p}^*$. Let's fix some $I$ corresponding to  the given $M$, $R_j (j\in M)$ and $1\leq x_{ji}\leq k_{ji} (i\in R_j,j\in M)$.
For any $j$, let $q_j=\# R_j$, and $R_j=\{c_{j1},\cdots, c_{jq_j}\}$ with $c_{j1}<c_{j2}<\cdots<c_{jq_j}$. For $j\in M$, set $\overline{x}_{jp}=x_{jc_{jp}}, \overline{k}_{jp}=k_{jc_{jp}}$ for $1\leq p\leq q_j$, $\overline{x}_{j0}=0$, and $d_j=\overline{x}_{jq_j}$.

\begin{Proposition}\label{obc} Given $I$ as above, then
$\mathrm{p}(g_I\cdot f)$  has depth $d=\sum_{j\in M}d_j$, and is $P$-conjugated to the element $\mathrm{pr}'(\eta)$, where
\[\eta=\mathrm{diag}(A_1,\cdots,A_m,J_d),\]
where $A_j=\xi_j$ if $j\notin M$; and $A_j=\mathrm{diag}(A_{j1},A_{j2},\cdots, A_{jr_j})$, if $j\in M$, where $A_{ji}$'s are defined as follows, $A_{ji}=J_{k_{ji},a_j}^{l_{ji}}$ if $i\notin R_j$, and
\[A_{ji}=\mathrm{diag}(J_{k_{ji},a_j}^{l_{ji}-1},J_{t,a_j})\]
with
\[t=\overline{k}_{jh}-\overline{x}_{jh}+\overline{x}_{j(h-1)},\]
if $i=c_{jh}\in R_j$ for  $1\leq h\leq q_j$.
\end{Proposition}

\begin{proof}
As in the proof of Proposition \ref{rg}, by block calculation, we can assume $l_{ji}=1$, $\forall 1\leq i\leq r_j$ and $R_j=\{1,\cdots,r_j\}$ for all $1\leq j\leq m$ at the beginning. Hence, $d_j=x_{jr_j}$ for $1\leq j\leq m$. For convenience, set $A_I=g_I\cdot \xi$.

In the case of $n\in I$,  we have $r_m=1$ since $\{x_{mi},i\in R_m\}$ satisfies condition (1) and (2).

We claim that $\mathrm{pr}'(A_I)$ is $P$-conjugated $\mathrm{pr}'(\zeta_j)$ for $0\leq j\leq m-1$, where
\[\zeta_j=\left(\begin{array}{cccccccc} A_1&&&&&&&\\ &J_{d_1,a_1}&&&&&&\\ &&\ddots&&&&&\\ &&&A_j&&&&\\ &&&&J_{d_j,a_j}&&&\\ &&&&& \xi_{j+1}&&\\ &&&&&& \ddots &\\ &&&&&&& \xi_m\end{array}\right)
+\left(\begin{array}{c}0_{(n-1)\times n}\\ \alpha_j\end{array}\right),\]
and
\[\alpha_j=(\beta_1',\cdots,\beta_j',\beta_{j+1},\cdots,\beta_m).\]
When $1\leq i\leq m-1$, $\beta'_i$ and $\beta_i$ are defined as follows.
\[\beta'_i=\left\{\begin{array}{ll}(\underbrace{0,\cdots,0}_{n_i-2},1,a_i-a_m), & d_i>1,\\  (\underbrace{0,\cdots,0}_{n_i-1},a_i-a_m), & d_i=1,\end{array}\right.\]
and $\beta_i=(\beta_{i1},\cdots,\beta_{ir_i})$, for $1\leq h\leq r_i$,
\[\beta_{ih}=\left\{\begin{array}{ll}
(\underbrace{0,\cdots,0}_{x_{ih}-2},1,a_i-a_m,\underbrace{0,\cdots,0}_{k_{ih}-x_{ih}}), & x_{ih}>1\\
(a_i-a_m,\underbrace{0,\cdots,0}_{k_{ih}-1}), & x_{ih}=1.\end{array}\right.\]
And $\beta_m=(\underbrace{0,\cdots,0}_{n_m})$.

It is trivially true that $\mathrm{pr}'(A_I)$ is $P$-conjugated $\mathrm{pr}'(\zeta_j)$ for $j=0$. Assume that it is true for $j-1$, let's prove that $\mathrm{pr}'(\zeta_{j-1})$ is $P$-conjugated $\mathrm{pr}'(\zeta_j)$, so we get that $\mathrm{pr}'(A_I)$ is $P$-conjugated $\mathrm{pr}'(\zeta_j)$.

Recall that $n_{j'}=\sum_{1\leq i\leq r_{j'}}k_{{j'}i}$ for $1\leq j'\leq m$, $n_0=0$. For convenience, set $r=r_j$, $k_u=k_{ju}$ for $1\leq u\leq r_j$, $x_h=x_{jh}$ for $1\leq h\leq r_j$, $x_0=0$.
Set
\[V=\{\sum_{j'=0}^{j-1}n_{j'}+1,\sum_{j'=0}^{j-1}n_{j'}+2,\cdots,\sum_{j'=0}^{j-1}n_{j'}+\sum_{u=1}^{r}k_u\}\]
and $U=V\cup\{n\}$.

We claim that one can use $P$-conjugations, which only change the elements in the positions $(U,V)$ or in the $n$-th column, to change $\zeta_{j-1}$ to $B$ such that
\[B(U,V)=\left(\begin{array}{c}\mathrm{diag}(J_{e_1}(a_j),\cdots,J_{e_r}(a_j),
J_{x_r}(a_j))\\ \beta_j'\end{array}\right), \]
where $e_h=k_h-x_h+x_{h-1}, 1\leq h\leq r$.

Step 1.
Let
\[P_1=\left(\begin{array}{ccccccc}
I_{n_1+\cdots+n_{j-1}}&&&&&&\\
&I_{k_1}&&&&&\\
&E_1&I_{k_2}&&&&\\
&&E_2&\ddots&&&\\
&&&\ddots&\ddots &&\\
&&&&E_{r-1}&I_{k_r}&\\
&&&&&&I_{n_{j+1}+\cdots+n_{m}}\end{array}\right),\]
where $E_{i}$ is the $k_{i+1}\times k_i$ matrix, with $-1$ in the positions $(x_{i+1},x_i), (x_{i+1}-1,x_i-1),\cdots,(x_{i+1}-x_i+1,1)$(here we use $x_{i+1}>x_{i}$) and $0$ in other positions, for $1\leq i\leq r-1$.

Let $B_1=P_1^{-1}\zeta_{j-1}P_1$, then
\[B_1(U,V) =\left(\begin{array}{ccccc}
J_{k_1,a_j} &&&& \\
E'_1 & J_{k_2,a_j} &&&\\
0& E'_2 & \ddots &&\\
\vdots& \ddots & \ddots &\ddots&\\
0&0&\cdots& E'_{r-1} & J_{k_r,a_j}\\
0&0&\cdots &0&\beta_{jr}
\end{array}\right),\]
where $E'_i$ is $k_{i+1}\times k_i$ matrix with $-1$ in the position $(x_{i+1}+1,x_i)$ and $0$ in other positions, for $1\leq i\leq r-1$.

Step 2.
Using the method of in  Proposition \ref{rg}, one can use $P$-conjugation, which just change the elements in the positions $(U,V)$ or in the $n$-th column, to change $B_1$ to $B_2$ such that
\[B_2(U,V)=\left(\setlength\arraycolsep{1mm}\begin{array}{ccccccc}
J_{k_1,a_j} &&&&& \\
 E'_1 & J_{k_2,a_j} &&&&\\
0& E'_2  & \ddots  &&&\\
\vdots &\ddots&   \ddots & \ddots &&\\
0&0&\ddots&     E'_{r-2} & J_{k_{r-1},a_j} &\\
0&0&\cdots&0&      0 & J_{x_r,a_j}&\\
0&0&\cdots&0&  E''_{r} &0& J_{k_r-x_r,a_j}\\
0&0&\cdots&0 &0&\beta_j'&0
\end{array}\right),\]
$E''_{r}$ is the $(k_r-x_r)\times k_{r-1}$ matrix with $1$ in the position $(1,x_{r-1})$ and $0$ in other positions.

Step 3.
Let
\[P_2=\left(\begin{array}{ccccccc}
I_{n_1+\cdots+n_{j-1}}&&&&&&\\
&I_{k_1}& F_1 &&&&\\
&&\ddots&\ddots&&&\\
&&&\ddots&F_{r-2}&&\\
&&&&I_{k_{r-1}}&F_{r-1}&\\
&&&&&I_{k_r}&\\
&&&&&&I_{n_{j+1}+\cdots+n_{m}}\end{array}\right),\]
$F_i$ is the $k_i\times k_{i+1}$ matrix with $-1$ in the positions $(x_i+1,x_{i+1}+1),\cdots,(k_i,x_{i+1}+k_i-x_i)$(here we use $k_{i+1}-x_{i+1}>k_i-x_{i}$), and $0$ in other positions, for $1\leq i\leq r$.

Let $B_3=P_2^{-1}B_2P_2$,  then $B_3(U,V)=$
\[\left(\setlength\arraycolsep{0mm}\begin{array}{ccccccccccccc}
J_{x_1,a_j} &&&&&&&& \\
0& J_{k_1-x_1,a_j} &&&&&&& \\
0& 0& J_{x_2,a_j} &&&&&&\\
E''_1&0&0&                J_{k_2-x_2,a_j}  &&&&&\\
0&0 &0&0 &J_{x_3,a_j} &&&&&&\\
0&0& E''_2&0&0&J_{k_3-x_3,a_j}&&&&\\
\vdots&\vdots&\vdots &\ddots& \ddots&\ddots&\ddots &&&&\\
\vdots&\vdots&\vdots &\ddots& \ddots&\ddots&\ddots&J_{x_{r-1},a_j}&&&\\
\vdots&\vdots&\vdots &\ddots& \ddots&\ddots&\ddots&0&J_{k_{r-1}-x_{r-1},a_j}&&\\
0&0&0&0&0  &0&\ddots &0& 0 & J_{x_r,a_j}&\\
0&0&0&    0&0&0&\cdots & E''_{r-1}&0&0& J_{k_r-x_r,a_j}\\
0&0&0&0&0&0&\cdots &0 &0&\beta_j' &0
\end{array}\right),\]
where $E''_i$ is the $(k_{i+1}-x_{i+1})\times x_i$ matrix with $1$ in the position $(1,x_i)$ and $0$ in other positions, for $1\leq i\leq r-1$.
Now we can easily see $B_3$ is $P$-conjugated to $\zeta_j$. This finishes the proof of the claim.

Applying Proposition \ref{rg}, by block calculation, we can easily show that $\mathrm{pr}'(\zeta_{m-1})$ is $P$-conjugated to $\mathrm{pr}'(\eta)$. This finish the proof of this case.

In the case of $n\notin I$, using similar method we can also get the statement.

This completes the proof of the statement.
\end{proof}

\begin{Theorem}\label{GLnC} We have that $\mathrm{p}:\mathcal{O}_f\to \mathfrak{p}^*$ sends different $P$-orbits of $\mathcal{O}_f$ to different $P$-orbits of $\mathfrak{p}^*$, the image $\mathrm{p}(\mathcal{O}_f)$ contains a unique dense orbit $P\cdot(g_{I^{o}}\cdot f)$, and $\mathrm{p}$ is proper over $P\cdot (g_{I^{o}}\cdot f)$.
The reduce space of the unique open dense orbit is singleton.
\end{Theorem}

\begin{proof}
The first statement and the last statement is from the Proposition \ref{obc}. Let's check the properness.

In the case of $m=1$, set $r=r_1$, $k_i=k_{1i}$, $l_i=l_{1i}$ for $1\leq i\leq r_1$.

As Lemma \ref{st}, $\mathrm{Stab}_P(g_{I^o}\cdot f)\cong \mathrm{Stab}_L(v_{I^{o}})$. For convenience, we denote $G_k(\mathbb{C})$ by $G_k$ for any $k\in \mathbb{Z}_+$.
We know that
\[L\cong (G_{l_1}\times G_{l_2}\times \cdots G_{l_r})\ltimes N(\sum_{i=1}^{r-1}k_il_i\cdot(2(l_{i+1}+\cdots+l_r))+\sum_{i=1}^rl_i^2(k_i-1)),\]
where $N(a)$ means the unipotent group which is homeomorphic to $\mathbb{C}^{a}$ as manifold.

And
\begin{align*}\mathrm{Stab}_L(v_{I^{o}})\cong& (G_{l_1}\times G_{l_2}\times \cdots G_{l_r-1})\ltimes N(\sum_{i=1}^{r-1}k_il_i\cdot(2(l_{i+1}+\cdots+l_r))\\&+\sum_{i=1}^rl_i^2(k_i-1)+(l_r-1)-(\sum_{i=1}^{r-1}k_il_i+(k_r-1)l_r)).
\end{align*}

Similarly, we have
\begin{align*}\mathrm{Stab}_P(\mathrm{p}(g_{I^{o}}\cdot f))\cong & (G_{l_1}\times G_{l_2}\times \cdots G_{l_r-1})\ltimes N(\sum_{i=1}^{r-1}k_il_i\cdot(2(l_{i+1}+\cdots+l_r-1))\\
+&\sum_{i=1}^{r-1}l_i^2(k_i-1)+(l_r-1)^2(k_r-1)+(\sum_{i=1}^{r-1}k_il_i+(l_r-1)k_r)).\end{align*}

By direct calculation, we have that the quotient
\[\mathrm{Stab}_P(\mathrm{p}(g_{I^{o}}\cdot f))/\mathrm{Stab}_P(g_{I^o}\cdot f)\]
is a single element. So the restriction of $\mathrm{p}$ on $P\cdot (g_{I^{o}}\cdot f)$ is proper.

In the case of $m>1$, one can get the statement immediately.
\end{proof}

\subsection{Moment map in the $\mathrm{GL}_n(\mathbb{R})$ case}

Now, we can  calculate the moment map in the $G_n(\mathbb{R})$ case and get the similar results as the case of $G_n(\mathbb{C})$.

Let $n\in \mathbb{Z}_+$, let $G=G_n(\mathbb{R})$ and let $P=P_n(\mathbb{R})$ be the mirabolic subgroup of $G$. $\mathfrak{g}$ (resp. $\mathfrak{p}$) denotes the Lie algebra of $G$ (resp. of $P$).

Write
\[J_s(a)=\left(\begin{array}{cccc} a& &&\\ 1&a &&\\ &\ddots &\ddots &\\ && 1& a\end{array}\right)\in \mathfrak{g}_s(\mathbb{R}), s\in \mathbb{Z}\ \text{and}\ s\geq 2, a\in \mathbb{R},  \]
and $J_1(a)=a_{1\times 1}, a\in \mathbb{R}$. Set $J_s=J_s(0)$, $\forall s\in \mathbb{Z}_+$.
Write
\[J_k^l(a)=\mathrm{diag}(\underbrace{J_k(a),\cdots,J_k(a)}_{l\ \text{blocks}}),\]
$\forall l\in \mathbb{Z}_+$.
Write
\[J_k(a,b)=\left(\begin{array}{cc} J_k(a) & bI_k\\ -bI_k & J_k(a)\end{array}\right), \forall  a,b\in \mathbb{R}, k\in \mathbb{Z}_+,\]
and
\[J_k^l(a,b)=\mathrm{diag}(\underbrace{J_k(a,b),\cdots,J_k(a,b)}_{l\ \text{blocks}}),\]

Let $u,v\in \mathbb{N}$ such that $2u+v\leq n$.
Let $a_1,\cdots,a_{u+v},b_1,\cdots,b_u\in \mathbb{R}$ such that $z_{2k-1}=a_k+ib_k$, $z_{2k}=a_k-ib_k$, for $k=1,\cdots, u$, and $z_{2u+k}=a_{u+k}$ for $1\leq k\leq v$, are $2u+v$ distinct complex number,  $\prod_{1\leq i\leq u}b_i\neq 0$.

Let
\[ \xi = \mathrm{diag}(\xi_1,\cdots,\xi_{u+v}),\]
with
\[\xi_j=\left\{\begin{array}{c}
\mathrm{diag}(J_{k_{j1}}^{l_{j1}}(a_j,b_j),J_{k_{j2}}^{l_{j2}}(a_j,b_j),\cdots,J_{k_{jr_j}}^{l_{jr_j}}(a_j,b_j)),   1\leq j\leq u,\\
\mathrm{diag}(J_{k_{j1}}^{l_{j1}}(a_j),J_{k_{j2}}^{l_{j2}}(a_j),\cdots,J_{k_{jr_j}}^{l_{jr_j}}(a_j)),  u+1\leq j\leq u+v,\end{array}\right.\]
where $r_j\in \mathbb{Z}_+, 1\leq j\leq u+v$, $l_{ji}, k_{ji}\in \mathbb{Z}_+$, $\forall 1\leq j\leq u+v, 1\leq i\leq r_j$, and $k_{ji_1}< k_{ji_2}$ for $1\leq j\leq {u+v}$, $1\leq i_1<i_2\leq r_i$, and
\[\sum_{1\leq j\leq u, 1\leq i\leq r_j}2k_{ji}l_{ji}+\sum_{u+1\leq j\leq u+v,1\leq i\leq r_j}k_{ji}l_{ji}=n.\]
Set $f=\mathrm{pr}(\xi)$. It is clear that every orbit in $\mathfrak{g}^*$ is of the form $\mathcal{O}_f$.

Now, we fix a $\xi$. Set
\[n_j=\left\{\begin{array}{ll}\sum_{i=1}^{r_j}2k_{ji}\cdot l_{ji}, & 1\leq j\leq u,\\
\sum_{i=1}^{r_j}k_{ji}\cdot l_{ji}, & u+1\leq j\leq u+v,\end{array}\right.\]
and $n_0=0$. For convenience, set $k_{j0}=0, l_{j0}=0$ for $1\leq j\leq u+v$.

Set $L=\mathrm{Stab}_G(f)$, then
\[P\setminus\mathcal{O}_f\cong P\setminus G/L \cong (\mathbb{R}^n-\{0\})/L.\]
By calculating the representatives of $(\mathbb{R}^n-\{0\})/L$, one can get the representatives of the $P$-orbits in $\mathcal{O}_f$.

Since $\xi$ is block diagonal,
\[L=\mathrm{diag}(X_1,\cdots,X_{u+v})\ |\ X_i\in\mathrm{Stab}_{\mathrm{GL}_{n_j}(\mathbb{R})}(\xi_{j})\}.\]

For $u+1\leq j\leq u+v$, one can obtain $L_j=\mathrm{Stab}_{\mathrm{GL}_{n_j}(\mathbb{R})}(\xi_{j})$ from the result in the case of $G_n(\mathbb{C})$ by changing the field $\mathbb{C}$ to $\mathbb{R}$ (see Proposition \ref{Lj}).

For $1\leq j\leq u$, let's calculate $L_j$.
Let $M_{s\times s}(\mathbb{R})$ be the set of all $s\times s$ matrices over $\mathbb{R}$. Write
\[K(x_1,\cdots,x_s)=\left(\begin{array}{cccc} x_1&&&\\ x_2& x_1 &&\\ \vdots& \ddots & \ddots &\\ \ x_s &\cdots &x_2&x_1\end{array}\right), x_i\in \mathbb{R}, 1\leq i\leq s,\]
then
\[\{X\in M_{s\times s}(\mathbb{R})\ |\ XJ_s=J_sX\}=\{K(x_1,\cdots,x_s)\ |\  x_i\in \mathbb{R}, 1\leq i\leq s\}.\]

For any $\vec{x}=(x_1,\cdots,x_s), \vec{y}=(y_1,\cdots,y_s) \in \mathbb{R}^s$, define $H(\vec{x},\vec{y})$ as the $2s\times 2s$ matrix
\[\left(\begin{array}{cc} K(x_1,\cdots,x_s) &K(y_1,\cdots,y_s) \\
-K(y_1,\cdots,y_s)& K(x_1,\cdots,x_s)\end{array}\right).\]
We have the following proposition.

\begin{Proposition}\label{Lkl} For any $a,b\in \mathbb{R}$, $b\neq 0$, $k,l\in \mathbb{Z}_+$,
\begin{align*}H_{k,l}&=\mathrm{Stab}_{\mathrm{GL}_{2k\cdot l}(\mathbb{R})}(J_k^l(a,b))\\
&=\{\left(\begin{array}{ccc}
H(\vec{x}_{11},\vec{y}_{11})&\cdots &H(\vec{x}_{1l},\vec{y}_{1l})\\
\vdots &\ddots &\vdots\\
H(\vec{x}_{l1},\vec{y}_{l1})&\cdots &H(\vec{x}_{ll},\vec{y}_{ll})
\end{array}\right) |\ \vec{x}_{ij}, \vec{y}_{ij}\in \mathbb{R}^{k}, \forall 1\leq i,j\leq l,
\\& \det\left(\begin{array}{cc}A&B\\ -B&A\end{array}\right)\neq 0,  A=(\vec{x}_{ij}(1))_{1\leq i,j\leq l}, B=(\vec{y}_{ij}(1))_{1\leq i,j\leq l}\}.\end{align*}
\end{Proposition}

\begin{proof}
It can be proved similarly as Proposition \ref{LK1}.
\end{proof}

For any $k_1,k_2 \in \mathbb{Z}_+$ such that $k_2\geq k_1$, and $\vec{x}=(x_1,\cdots,x_{k_1})$, $\vec{y}=(y_1,\cdots,y_{k_1})\in \mathbb{R}^{k_1}$, define $H_1(\vec{x},\vec{y})$ as the $2k_1\times 2k_2$ matrix
\[\left(\begin{array}{cccc} K(x_1,\cdots,x_{k_1})& 0_{k_1\times (k_2-k_1)} &K(y_1,\cdots,y_{k_1})&0_{k_1\times (k_2-k_1)}\\
-K(y_1,\cdots,y_{k_1})& 0_{k_1\times (k_2-k_1)} &K(x_1,\cdots,x_{k_1})&0_{k_1\times (k_2-k_1)}\end{array}\right),\]
and define $H_2(\vec{x},\vec{y})$ as the $2k_2\times 2k_1$ matrix
\[\left(\begin{array}{cc} 0_{(k_2-k_1)\times k_1}& 0_{(k_2-k_1)\times k_1}\\
 K(x_1,\cdots,x_{k_1})&K(y_1,\cdots,y_{k_1})\\
 0_{(k_2-k_1)\times k_1}& 0_{(k_2-k_1)\times k_1}\\
-K(y_1,\cdots,y_{k_1})& K(x_1,\cdots,x_{k_1})\end{array}\right).\]

\begin{Proposition}
For any $a,b\in \mathbb{R}$, $b\neq 0$, $k_1,k_2,l_1,l_2\in \mathbb{Z}_+$, $k_1<k_2$,
\begin{align*}M_{k_1,k_2}^{l_1,l_2}&=\{X\in M_{2k_1l_1\times 2k_2l_2}(\mathbb{R})\ |\ X(J_{k_2}^{l_2}(a,b))=(J_{k_1}^{l_1}(a,b))X\}\\
&=\{\left(\begin{array}{ccc}H_1(\vec{x}_{11},\vec{y}_{11})&\cdots &H_1(\vec{x}_{1l_2},\vec{y}_{1l_2})\\
\vdots &\ddots &\vdots\\
H_1(\vec{x}_{l_11},\vec{y}_{l_11})&\cdots &H_1(\vec{x}_{l_1l_2},\vec{y}_{l_1l_2})\end{array}\right)\\
& |\ \vec{x}_{ij},  \vec{y}_{ij} \in \mathbb{R}^{k_1}, \forall 1\leq i\leq l_1, 1\leq j\leq l_2\},\end{align*}
\begin{align*}N_{k_1,k_2}^{l_1,l_2}&=\{X\in M_{2k_2l_2\times 2k_1l_1}(\mathbb{R})\ |\ X(J_{k_1}^{l_1}(a,b))=(J_{k_2}^{l_2}(a,b))X\}
\\&=\{\left(\begin{array}{ccc}H_2(\vec{x}_{11},\vec{y}_{11})&\cdots &H_2(\vec{x}_{1l_1},\vec{y}_{1l_1})\\
\vdots &\ddots &\vdots\\
H_2(\vec{x}_{l_21},\vec{y}_{l_21})&\cdots &H_2(\vec{x}_{l_2l_1},\vec{y}_{l_2l_1})\end{array}\right)\\&
 |\ \vec{x}_{ij}, \vec{y}_{ij} \in \mathbb{R}^{k_1}, \forall 1\leq i\leq l_2, 1\leq j\leq l_1\}.\end{align*}
\end{Proposition}

\begin{proof}
It is easy to check.
\end{proof}

For convenience, we fix a $1\leq j\leq u$ and set $r=r_j$, $k_i=k_{ji}, l_i=l_{ji}$, for $1\leq i\leq r_j$ in the Proposition \ref{Ljr}.

\begin{Proposition}\label{Ljr} The stabilizer
\begin{align*}L_{j}&=\mathrm{Stab}_{\mathrm{GL}_{n_j}(\mathbb{R})}(\xi_{j})=
\{\left(\begin{array}{cccc}
X_{11}&X_{12}&\cdots &X_{1r}\\
X_{21}&X_{22}&\cdots &X_{2r}\\
\vdots &\vdots &\ddots &\vdots\\
X_{r1}&X_{r2}&\cdots &X_{rr}
\end{array}\right)
 \ |\ X_{ss}\in H_{k_s,l_s}, \forall 1\leq s\leq r,
\\& X_{st}\in M_{k_s,k_t}^{l_s,l_t}, \forall 1\leq s<t\leq r,
X_{st}\in N_{k_s,k_t}^{l_s,l_t}, \forall 1\leq t<s\leq r\}. \end{align*}
\end{Proposition}
\begin{proof}
It can be checked directly as Proposition \ref{Lj}.
\end{proof}

Now we go to calculate $(\mathbb{R}^n-\{0\})/L$. As $L$ is block diagonal, we only need to calculate $(\mathbb{R}^{n_j}-\{0\})/L_j$. For $u+1\leq j\leq u+v$, we can get the result from the case of $G_n(\mathbb{C})$. For $1\leq j\leq u$, we have the following result.

We fix a $j$ and adopt the same notation as Proposition \ref{Ljr}.

\begin{Proposition}\label{dbr}
Write $R_0=\{1,\cdots,r\}$. For $\emptyset \neq R\subset R_0$, let $1\leq x_i\leq k_i$, for $i\in R$, such that
\begin{itemize}
\item[(1)] $x_i < x_{i'}$ for any $i,i'\in R$ such that  $i<i'$,
\item[(2)] $k_i-x_i < k_{i'}-x_{i'}$ for any $i,i'\in R$ such that $i<i'$,
\end{itemize}
we define
\[I=\{\sum_{t=0}^{i-1}2k_tl_t+k_i(2l_i-1)+x_i, i\in R\}\subseteq \{1,2,\cdots, n_j\},\]
and $v_I=(y_1,y_2,\cdots,y_{n_j})$, $y_i=1$ when $i\in I$ and $y_i=0$ when $i\notin I$. Let $I_a$ be the set of all $I$'s constructed as above.
Then $\{v_I\ |\ I\in I_a\}$ represent all different $L_j$-orbits in $\mathbb{R}^{n_j}-\{0\}$.
\end{Proposition}
\begin{proof}
It can be proved similarly as Proposition \ref{db}.
\end{proof}

Now we can get the representatives of $(\mathbb{R}^n-\{0\})/L$.

\begin{Proposition}
Set $K_0=\{1,\cdots, u+v\}$, $R_{j0}=\{1,\cdots, r_j\},\forall 1\leq j\leq u+v$. Let $\emptyset\neq K\subseteq K_0$, and $\emptyset\neq R_j\subseteq R_{j0}$ for any $j\in K$. Choose an integer $1\leq x_{ji}\leq k_{ji}$ for $j\in K, i\in R_j$, such that
\begin{itemize}
\item[(1)] $x_{j{i_1}}<x_{j{i_2}}$, for any $i_1,i_2\in R_j$ such that $i_1<i_2$,
\item[(2)] $k_{j{i_1}}-x_{j{i_1}}<k_{j{i_2}}-x_{j{i_2}}$, for any $i_1,i_2\in R_j$ such that $i_1<i_2$.
\end{itemize}
 Set
\begin{align*}I=&\bigsqcup_{j\in K,1\leq j\leq u}\{\sum_{l=0}^{j-1}n_l+\sum_{p=0}^{i-1}2k_{jp}l_{jp}+k_{ji}(2l_{ji}-1)+x_{ji}, i\in R_j\}\bigsqcup \\&\bigsqcup_{j\in K,u+1\leq j\leq u+v}\{\sum_{l=0}^{j-1}n_l+\sum_{p=0}^{i-1}k_{jp}l_{jp}+k_{ji}(l_{ji}-1)+x_{ji}, i\in R_j\},\end{align*}
and $v_I=(v_1,\cdots,v_n)$ with $v_i=1$ when $i\in I$ and $v_i=0$ when $i\notin I$. Let $I_a$ be the set of all $I$'s constructed above, then $\{v_I\ |\ I\in I_a\}$ form all different representatives of $(\mathbb{R}^n-\{0\})/L$.
\end{Proposition}

\begin{proof}
It can be proved similarly as Proposition \ref{db1}.
\end{proof}

We can define $g_I$ in the same way as the contents about regular orbits and such $g_I\cdot f$'s form all different representatives of $P\setminus\mathcal{O}_f$ and the orbit of $g_{I^o}\cdot f$, $I^{o}=\{n_1,n_1+n_2,\cdots,n\}$, is the unique dense open orbit among them.

We go to calculate the moment map $\mathrm{p}: \mathcal{O}_f\to \mathfrak{p}^*$. Let's fix some $I$ corresponding to  the given $K$, $R_j (j\in K)$ and $1\leq x_{ji}\leq r_j (i\in R_j,j\in K)$.
For any $1\leq j\leq u+v$, let $q_j= \# R_j$, and $R_j=\{c_{j1},\cdots, c_{jq_j}\}$, $c_{j1}<c_{j2}<\cdots<c_{jq_j}$. Set $\overline{x}_{jh}=x_{jc_{jh}}, \overline{k}_{jh}=k_{jc_{jh}}, \overline{l}_{jh}=l_{jc_{jh}}$ for $1\leq h\leq q_j$, set $d_j=\overline{x}_{jq_j}$ and $\overline{x}_{j0}=0$.

\begin{Proposition}\label{obr}
We have that $\mathrm{p}(g_I\cdot f)$ has depth
\[d=\sum_{j\in M,1\leq j\leq u}2d_j+\sum_{j\in M,u+1\leq j\leq u+v}d_j,\]
and is $P$-conjugated to the element $\mathrm{pr}'(\eta)$, where
\[\eta=\mathrm{diag}(A_1,\cdots,A_m,J_d),\]
where $A_j=\xi_j$ if $j\notin M$, and $A_j=\mathrm{diag}(A_{j1},A_{j2},\cdots, A_{jr_j})$, if $j\in M$, where $A_{ji}$'s are defined as follows, when $i\notin R_j$,
\[A_{ji}=\left\{\begin{array}{ll}J_{k_{ji}}^{l_{ji}}(a_j,b_j), &1\leq j\leq u,
\\ J_{k_{ji}}^{l_{ji}}(a_j), &u+1\leq j\leq u+v,\end{array}\right.\]
when $i=c_{jh}\in R_j$ for some $1\leq h\leq q_j$,
\[A_{ji}=\left\{\begin{array}{ll}\mathrm{diag}(J_{k_{ji}}^{l_{ji}-1}(a_j,b_j),J_t(a_j,b_j)),&1\leq j\leq u,\\ \mathrm{diag}(J_{k_{ji}}^{l_{ji}-1}(a_j),J_t(a_j)),& u+1\leq j\leq u+v,\end{array}\right.\]
here $t=k_{ji}-\overline{x}_{jh}+\overline{x}_{j(h-1)}$.
\end{Proposition}

\begin{proof}
We regard $\mathfrak{p}_n(\mathbb{R})$ as the real form of $\mathfrak{p}_n(\mathbb{C})$. Use the result of $G_n(\mathbb{C})$, we get that $\mathrm{p}(g_I\cdot f)$ is $P_n(\mathbb{C})$-conjugated to $\mathrm{pr}'(\eta)$.
So $\mathrm{p}(g_I\cdot f)$ is also $P_n(\mathbb{R})$-conjugated to $\mathrm{pr}'(\eta)$.
\end{proof}

\begin{Theorem}\label{GLnR} $\mathrm{p}:\mathcal{O}_f\to \mathfrak{p}^*$ sends different $P$-orbits of $\mathcal{O}_f$ to different $P$-orbits of $\mathfrak{p}^*$, the image $p(\mathcal{O}_f)$ contains a unique dense orbit $P\cdot(g_{I^{o}}\cdot f)$, and the restriction of $\mathrm{p}$ on $P\cdot (g_{I^{o}}\cdot f)$ is proper.
The reduce space of the unique open dense orbit is singleton.
\end{Theorem}

\begin{proof}
It follows from Proposition \ref{obr} and the similar result of the $G_n(\mathbb{C})$ case (see Theorem \ref{GLnC}).
\end{proof}

\section{Results of Kirillov's conjecture}\label{sec-4}

Let $k$ be the field $\mathbb{R}$ or $\mathbb{C}$, let $P_n(k)$ be the mirabolic subgroup of $\mathrm{GL}_n(k)$. The Kirillov's conjecture states that the irreducible unitary representations of $\mathrm{GL}_n(k)$ remains irreducible upon restriction to $P_n(k)$ and was proved by Sahi (\cite{S}) for tempered representations of $\mathrm{GL}_n(\mathbb{R})$ or $\mathrm{GL}_n(\mathbb{C})$, and
Sahi-Stein (\cite{SS}) for Speh representations of $\mathrm{GL}_n(\mathbb{R})$, and Baruch
(\cite{Bar}) in archimedean fields general. Later Aizenbud-Gourevitch-Sahi (\cite{SS}) calculated the adduced representations of the Speh complementary series.

Sahi's and Sahi-Stein's proofs are based on Mackey's theory of the unitary representations of semi-direct products and Vogan's result about the classification of the irreducible unitary representations of $\mathrm{GL}_n(k)$, and their key method is constructing the intertwining operator by Fourier transform. Baruch proved by studying the $P_n(k)$-invariant distribution. Aizenbud-Gourevitch-Sahi's used annihilator variety and degenerate Whittaker models to get the adduced representations of the Speh complementary series. Here we summarize some results of Kirillov's conjecture.

Firstly, we show how to get the unitary dual of $P_n(k)$, denoted by $\widehat{P_n(k)}$. The following two facts are easy to see.
\begin{itemize}
\item[(i)]  $P_n(k)\cong \mathrm{GL}_{n-1}(k)\ltimes k^{n-1}$.
\item[(ii)] $\mathrm{GL}_{n-1}(k) $ has two orbits in $(k^{n-1})^*$: $\{0\}$ and $(k^{n-1})^*-\{0\}$. Let $\xi\in (k^{n-1})^*-\{0\}$ be defined by $\xi(x_1,\cdots,x_{n-1})=x_{n-1}$. Then $\mathrm{Stab}_{\mathrm{GL}_{n-1}(k)}(\xi)\cong P_{n-1}(k)$.
\end{itemize}
Based on Mackey's theory, we have the followings.

Every irreducible unitary representation of $P_n(k)$ is obtained in one of the following two ways:
\begin{itemize}
\item[(i)] by trivially extending an irreducible unitary representation of $\mathrm{GL}_{n-1}(k)$,
\item[(ii)] by extending an irreducible unitary representation of $P_{n-1}(k)$ to $P_{n-1}(k)\ltimes k^{n-1}$ by the character $\xi$ and then unitarily inducing to $P_n(k)$.
\end{itemize}

We use $E$ and $I$ to denote for functors from the above constructions
(i) and (ii) respectively. Then,
\[\widehat{P_n(k)}=E(\widehat{\mathrm{GL}_{n-1}(k)})\bigsqcup I(\widehat{P_{n-1}(k)}).\]
Moreover, we have the following fact: each irreducible unitary representation
$\tau$ of $P_n(k)$ is of the form $\tau= I^{j-1}E\sigma$, where the integer $j\geq 1$ and $\sigma\in \widehat{\mathrm{GL}_{n-1}(k)}$ are uniquely determined by $\tau$.

Now we describe the unitary dual of $\mathrm{GL}_n(k)$, which is finally obtained by Vogan \cite{V1}. Let $\pi_1$ and $\pi_2$ be the representations of $\mathrm{GL}_{n_1}(k)$ and $\mathrm{GL}_{n_2}(k)$, $\pi_1\times \pi_2$ denotes the unitary parabolic induction from the representation $\pi_1\otimes \pi_2$ of the Levi subgroup $\mathrm{GL}_{n_1}(k)\times \mathrm{GL}_{n_2}(k)$ to $\mathrm{GL}_{n_1+n_2}(k)$. Then
every irreducible unitary representation of $\mathrm{GL}_n(\mathbb{R})$ is a $\times$-product of unitary characters,
Stein representations, Speh representations, and Speh complementary series representations.
And every irreducible unitary representation of $\mathrm{GL}_n(\mathbb{C})$ is a $\times$-product of
unitary characters and Stein representations.

For convenience, we give the descriptions of Speh representations, Stein representations and Speh complementary series representations.
They can be described as the (subrepresentations of) degenerated principle series.
\begin{itemize}
\item[(i)]Speh representations.
Let $Q$ be the subgroup of $\mathrm{GL}_{2n}(\mathbb{R})$ defined as
\[Q=\{\left(\begin{array}{cc}A&B\\0&D\end{array}\right)\ |\ A,D\in \mathrm{GL}_n(\mathbb{R}), B\in M_{n\times n}(\mathbb{R})\}.\]
For any $m\in \mathbb{Z}_+$, we define a character $\chi_m: Q\to \mathbb{C}^{\times}$  as
\[ \left(\begin{array}{cc}A&B\\0&D\end{array}\right)\mapsto |\det(A)|^{m/2}\mathrm{sgn}(\det(A))^{m+1}|\det(D)|^{-m/2}.\]
The Speh representation of $\mathrm{GL}_{2n}(\mathbb{R})$, denoted by $\delta(n,m)$, is the unique nonzero irreducible subrepresentation of $\mathrm{Ind}_Q^G(\chi_m)$, where the smooth vector in $\mathrm{Ind}_Q^G(\chi_m)$ is
\[\{f\in C^{\infty}(G)\ |\ f(qg)=\lambda(q)\chi_m(q)f(g), \forall q\in Q, g\in G\},\]
here $\lambda$ is the modular character defined as \[\left(\begin{array}{cc}A&B\\0&D\end{array}\right)\mapsto|\det(A)|^{n/2}|\det(D)|^{-n/2}.\]

Simply speaking, the Speh representation $\delta(n,m)$ is the unique nonzero irreducible subrepresentation of
\[(|\det|^{m/2}\cdot \mathrm{sgn}(\det)^{m+1})|_{\mathrm{GL}_n(\mathbb{R})}\times |\det|^{-m/2}|_{\mathrm{GL}_n(\mathbb{R})}.\]
\item[(ii)] Stein representation. For any  $s\in (0,1/2)$, we can define
the Stein representation of $\mathrm{GL}_{2n}(\mathbb{R})$, denoted by $\sigma(n,s)$ (or $\sigma(n,s)|_{\mathrm{GL}_{2n}(\mathbb{R})}$ when there is ambiguity), as the representation
\[|\det|^{s}|_{\mathrm{GL}_n(\mathbb{R})}\times |\det|^{-s}|_{\mathrm{GL}_n(\mathbb{R})}.\]
The Stein representation of $\mathrm{GL}_{2n}(\mathbb{C})$, denoted by $\sigma(n,s)$(or $\sigma(n,s)|_{\mathrm{GL}_{2n}(\mathbb{C})}$ when there is ambiguity), can be defined the same as that of $\mathrm{GL}_{2n}(\mathbb{R})$,
\[|\det|^{s}|_{\mathrm{GL}_n(\mathbb{C})}\times |\det|^{-s}|_{\mathrm{GL}_n(\mathbb{C})}.\]
\item[(iii)]Speh complementary series representation.
For any $m\in \mathbb{Z}_+$, and $s\in (0,1/2)$, we can define the Speh complementary series representation of $\mathrm{GL}_{4n}(\mathbb{R})$, denoted by $\Delta(n,m,s)$, as the representation
\[ |\det|^{s}\delta(n,m) \times |\det|^{-s}\delta(n,m).\]
\end{itemize}

We now state some results on Kirillov's conjecture. Firstly, S.Sahi \cite{S} established Kirillov's conjecture of $\mathrm{GL}_n(k)$ for the $\times$-product of unitary characters and Stein representations.

Sahi \cite{S} defined that
\begin{itemize}
\item[(i)]
a unitary representation of $P_n(k)$, denoted by $\tau$, is homogeneous of depth $j$ if $\tau=I^{j-1}E\sigma$ for some unitary representation $\sigma$ of $\mathrm{GL}_{n-1}(k)$,
\item[(ii)]
a unitary representation of $\mathrm{GL}_n(k)$, denoted by $\rho$, is adducible of depth $j$ if $\rho|_{P_n(k)}$ is homogeneous of depth $j$, and  if $\rho|_{P_n(k)}=I^{j-1}E\sigma$, we shall write $\sigma=A\rho$ and call it the adduced representation of $\rho$.
\end{itemize}

Then Sahi obtained the following key fact based on Mackey's theory and the partial Fourier transform, and got the result in that paper.
\begin{Theorem}[\cite{S}, Theorem 2.1] If $\rho$ and $\sigma$ are adducible representation of $\mathrm{GL}_{r}(k)$ and $\mathrm{GL}_{s}(k)$ of depths $l$ and $m$, then $\rho\times \sigma$ is adducible of depth $l+m$. Moreover, $A(\rho\times\sigma)=(A\rho) \times (A\sigma)$.
\end{Theorem}
We also have the following result.
\begin{Theorem}[\cite{S}, Lemma 3.1]  Let $\pi$ be a unitary character of $\mathrm{GL}_n(k)$, then $\pi$ is adducible of depth $1$ and $A\pi=\pi|_{\mathrm{GL}_{n-1}(k)}$, where $\mathrm{GL}_{n-1}(k)$ is imbedded on the top left corner of $\mathrm{GL}_n(k)$.
\end{Theorem}

Latter, Sahi \cite{S1}, Sahi-Stein \cite{SS} obtained the adduced representations of Stein representations and Speh representations. The key method is constructing the intertwining operator by partial Fourier transform.

\begin{Theorem}[\cite{S1}, 2.4]
Let $\sigma(n+1,s)$ be the Stein representation of $\mathrm{GL}_{2n+2}(\mathbb{R})$, $s\in (0,1/2)$, then $A\sigma(n+1,s)=\sigma(n,s)$.
\end{Theorem}

\begin{Theorem}[\cite{SS}, Theorem 3]
Let $\delta(n+1,m)$ be the $m$-th Speh representation of $\mathrm{GL}_{2n+2}(\mathbb{R})$, $m\in \mathbb{Z}_+$, then $A\delta(n+1,m)=\delta(n,m)$.
\end{Theorem}

After Baruch \cite{Bar}  proved the Kirillov's conjecture in archimedean fields, Aizenbud-Gourevitch-Sahi \cite{SS} calculated the adduced representations of the Speh complementary series.

\begin{Theorem}[\cite{AGS}, Theorem 4.2.4]
Let $\Delta(n+1,m,s)$ be the Speh complementary  series representation of $\mathrm{GL}_{4n+4}(\mathbb{R})$, $m\in \mathbb{Z}_+$ and $s\in (0,1/2)$, then $A\Delta(n+1,m,s)=\Delta(n,m,s)$.
\end{Theorem}

\begin{Example}
In the case of $\mathrm{GL}_n(\mathbb{C})$, let $a_i\in \mathbb{R}, s_i\in \mathbb{Z}_+$, $1\leq i\leq m$, $\sum_{i=1}^ms_i=n$, let $\pi=(\det)^{ia_1}|_{\mathrm{GL}_{s_1}(\mathbb{C})}\times \cdots \times (\det)^{ia_{m}}|_{\mathrm{GL}_{s_m}(\mathbb{C})}$ be the irreducible unitary representation of $\mathrm{GL}_n(\mathbb{C})$, then
\[\pi|_{P_n}=I^{m-1}E((\det)^{ia_1}|_{\mathrm{GL}_{s_1-1}(\mathbb{C})}\times \cdots \times (\det)^{ia_{m}}|_{\mathrm{GL}_{s_m-1}(\mathbb{C})}).\]
\end{Example}

\section{Orbit method}\label{sec-5}

The orbit method tries to establish a close connection existed between irreducible unitary representations of a Lie group
and its orbits in the coadjoint representation, and to provide a clear geometric picture of irreducible unitary representations. In 1950's , before the orbit method developed, Mackey and others had obtained some wonderful results on the irreducible unitary representations of Lie groups, showing how to use induced representation to obtain irreducible unitary representations, see Mackey \cite{M}. In 1960's, Kirillov found out a way to think about the set of all irreducible unitary representations of a simply connected nilpotent Lie group, establishing a bijection between the coadjoint orbits and the irreducible unitary representations (see Kirillov \cite{K}), and brought up the orbit method. Later, Kostant developed a quantization theory to obtain irreducible unitary representations, and generalized the results of Kirillov to solvable Lie groups (see Auslander-Kostant \cite{AK}).  In 1970's-1980's,  Duflo got all the irreducible unitary representations of (almost) algebraic groups assuming the unitary dual of reductive Lie groups, see Duflo \cite{D}. The orbit method (or the quantization problem) of reductive groups is not completely solved until now. Roughly speaking, the quantization of coadjoint orbits can be reduced to the case of nilpotent coadjoint orbits (defined below), and such irreducible unitary representations attached to nilpotent coadjoint orbits, called unipotent representations, are not easy to understood in general, see Vogan \cite{V}. We will summarize some results about the orbit method of reductive Lie groups and algebraic groups in the following two parts. The method of constructing representations is using induced representation (includes parabolic induction and cohomological induction).

\subsection{Representations attached to coadjoint orbits of $\mathrm{GL}_n(k)$}

\subsubsection{Jordan decomposition}

To begin with, we use Jordan decomposition to describe the coadjoint orbits of reductive groups (see Vogan \cite{V}, Lecture 2).

Let $G$ be a reductive Lie group and let $\mathfrak{g}$ be its Lie algebra. Let $\theta$ be a Cartan involution of $\mathfrak{g}$ and let
\[G=K\exp \mathfrak{s}\]
be the Cartan decomposition, where $K=G^{\theta}$ and $\mathfrak{s}=-1$ eigenspace of $\theta$ on $\mathfrak{g}$.

Let $B$ be a nondegenerate $G$-invariant and $\theta$-invariant bilinear form on $\mathfrak{g}$ such that the quadratic form
\[\mathfrak{g}\ni X\mapsto -B(X,\theta X)\]
on $\mathfrak{g}$ is positive definite.
Let $\mathfrak{g}^*$ be the dual of $\mathfrak{g}$. Then we have a $G$-isomorphism
\[\phi: \mathfrak{g} \to \mathfrak{g}^*, X\mapsto (Y\mapsto B(X,Y), \forall Y\in\mathfrak{g}), \forall X\in \mathfrak{g}.\]
For any $\lambda\in \mathfrak{g}^*$, we call it semisimple (resp. nilpotent, hyperbolic, elliptic)
if $\phi^{-1}(\lambda)$ is a semisimple (resp. nilpotent, hyperbolic, elliptic) element in $\mathfrak{g}$.
By the Jordan decomposition, we can decompose any $\lambda\in \mathfrak{g}^*$ uniquely as
\[\lambda=\lambda_h+\lambda_e+\lambda_n,\]
where $\lambda_h$ is hyperbolic, $\lambda_e$ is elliptic, $\lambda_n$ is nilpotent and $\phi^{-1}(\lambda_h)$,
$\phi^{-1}(\lambda_e)$, $\phi^{-1}(\lambda_n)$ commute with each other.

Set $X_h=\phi^{-1}(\lambda_h)$, $X_e=\phi^{-1}(\lambda_e)$ and $X_n=\phi^{-1}(\lambda_n)$.
Let $G(X)$ (resp. $G(\lambda)$) be the stabilizer of $X$ (resp. of $\lambda$), and $\mathfrak{g}(X)$ (resp. $\mathfrak{g}(\lambda)$) be the Lie algebra of $G(X)$ (resp. of $G(\lambda)$). Then we have the followings.
\begin{Proposition}[Vogan \cite{V1}, Proposition 2.10, 2.11 and 2.12]
\begin{itemize}
 \item[(i)]  Any hyperbolic element in $\mathfrak{g}$ is conjugated into $\mathfrak{s}$. If $X\in \mathfrak{s}$ is hyperbolic, then $G(X)$ is reductive group with Cartan involution $\theta|_{G(X)}$.
 \item[(ii)] There is a bijection between the $G$-orbits of elements in $\mathfrak{g}$, which hyperbolic part are $G$-conjugated to $X_h$, and the $G(X_h)$-orbits of elements in $\mathfrak{g}(X_h)$, which hyperbolic part is zero.
 \item[(iii)] Any elliptic element in $\mathfrak{g}$  is conjugated into $\mathfrak{k}$. If $X\in \mathfrak{k}$ is elliptic, then $G(X)$ is a reductive group with Cartan involution $\theta|_{G(X)}$.
 \item [(iv)] There is a bijection between the $G$-orbits of elements in $\mathfrak{g}$, which semisimple part are $G$-conjugated to $X_h+X_e$, and the $G(X_h+X_e)$-orbits of elements in $\mathfrak{g}(X_h+X_e)$, which semisimple part is zero.
\end{itemize}
\end{Proposition}
Therefore we can assume $X_h\in \mathfrak{s}$ and $X_e\in \mathfrak{k}$.

An irreducible unitary representation which is attached to a coadjoint orbit has a kind of ``Jordan decomposition". To attach an irreducible unitary representation to the orbit of $\lambda$,
we firstly attach an irreducible representation $\pi_n$ of $G(\lambda_h+\lambda_e)$ to the $G(\lambda_h+\lambda_e)$-orbit of $\lambda_n|_{\mathfrak{g}(\lambda_h+\lambda_e)^*}$, which is called as nilpotent step; then using $\lambda_e$ and applying cohomological induction to $\pi_n$, we get an irreducible unitary representation $\pi_h$ of $G(\lambda_h)$, which is called as elliptic step; finally, using $\lambda_h$ and applying parabolic induction to $\pi_h$, we get an irreducible unitary representation of $G$, which is called as hyperbolic step.

\subsubsection{Hyperbolic step}

The easiest step is hyperbolic step, which uses parabolic induction, see Vogan \cite{V}, Lecture 2.

Set $L=G(\lambda_h)$ with its Lie algebra $\mathfrak{l}$. Set
\[\mathfrak{g}_r=\{Y\in\mathfrak{g}|[X_h,Y]=rY\}, r\in \mathbb{R}\]
and $\mathfrak{u}=\oplus_{r>0}\mathfrak{g}_r$. Set $U=\exp \mathfrak{u}$ and $Q=LU$. Define unitary character $\chi(\lambda_h)$ of $L$ as follows,
\[ \chi(\lambda_h)(k\cdot \exp(Z))=\exp(i\lambda_h(Z)), k\in L\cap K, Z\in \mathfrak{l}\cap \mathfrak{s}.\]
Suppose that $\pi_L$ is any unitary representation of $G(\lambda_h)$ attached to the $G(\lambda_h)$-coadjoint orbit of $(\lambda_e+\lambda_n)|_{\mathfrak{l}}$, then the unitary representation of $G$ (this unitary representation may be reducible)
\[\pi_G=\mathrm{Ind}_Q^G(\pi_L\otimes \chi(\lambda_h))\]
is attached to the coadjoint orbit of $\lambda$.

\subsubsection{Elliptic step}

The elliptic step is more complicated, which uses cohomological induction, see Vogan \cite{V2}. Set $\lambda_s=\lambda_h+\lambda_e$ be semisimple part of $\lambda$, and set $X_s=X_h+X_e$.

Let $\mathfrak{g}_{\mathbb{C}}$ be the complexification of $\mathfrak{g}$, and let $\mathfrak{q}$ be the parabolic subalgebra having $\mathfrak{g}_{\mathbb{C}}(\lambda_s)$ as a Levi factor with nilpotent radical $\mathfrak{u}$ such that the eigenvalues of $X_s$ acting on $\mathfrak{u}$ are in
\[\{z\in \mathbb{C}\ |\ \mathrm{Re}(z)>0, \ \text{or}\ \mathrm{Re}(z)=0, \mathrm{Im}(z)>0\}.\]
Let $e^{2\rho(\mathfrak{u})}$ be the character of the adjoint action of $G(\lambda_s)$ on the top exterior power of $\mathfrak{u}$.

Assume that $\tau$ is any irreducible representation of $G(\lambda_s)$ such that \[d\tau=i\lambda_s+\rho(\mathfrak{u}).\]
Then there is
attached to $\lambda$  a unitary representation $\pi_{\lambda}$ such that the underlying $(\mathfrak{g}_{\mathbb{C}},K)$-module $\Pi(\lambda)$ is defined as follows,
\begin{equation}\label{pi_e}
\Pi(\lambda)=(\Gamma^{\mathfrak{g}_{\mathbb{C}},K}_{\mathfrak{g}_{\mathbb{C}},G(\lambda_s)\cap K})^d
(\mathrm{pro}_{\mathfrak{q},G(\lambda_s)\cap K}^{\mathfrak{g}_{\mathbb{C}},G(\lambda_s)\cap K}(\tau\otimes\pi_n)),
\end{equation}
where $\Gamma$ is the Zuckerman functor, $\mathrm{pro}$ is a kind of $\mathrm{Hom}$ functor, $d$ is the dimension of $\mathfrak{u}\cap \mathfrak{k}$, $\pi_n$ is the unipotent representation of $G(\lambda_s)$ which is attached to $\lambda_n|_{(\mathfrak{g}(\lambda_s))^*}$.

To define the Zuckerman functor (see Knapp-Vogan \cite{KV}, Introduction), we need to define the Hecke algebra. Assume that $\mathfrak{h}$ is a complex reductive Lie algebra of Lie group $H$, $M$ is a compact subgroup of $H$, then the Hecke algebra $R(\mathfrak{h},M)$ can be defined as the algebra of left $M$-finite distributions on $H$ with support in $M$, with convolution as multiplication.

Let's fix some notations to define the Zuckerman functor and the functor $\mathrm{pro}$. Let $G$ be a linear connected reductive Lie group and $G_{\mathbb{C}}$ be the complexification of $G$. Let $K$ be the maximal compact subgroup of $G$.  Let $\mathfrak{g}$ denote the Lie algebra of $G$ and $\mathfrak{g}_{\mathbb{C}}$ be the complexification of $\mathfrak{g}$.

Assume that $T$ is a torus in $G$, and $L=Z_G(T)$ with Lie algebra $\mathfrak{l}$. Let $L_{\mathbb{C}}$ be the analytic subgroup of $G_{\mathbb{C}}$ with Lie subalgebra $\mathfrak{l}_{\mathbb{C}}$.
Let $Q$ be a parabolic subgroup in $G_{\mathbb{C}}$ containing $L_{\mathbb{C}}$ as Levi subgroup and let $\mathfrak{q}$ be the Lie algebra of $Q$.

Assume that $V$ is a $(\mathfrak{g}_{\mathbb{C}},L\cap K)$-module, $W$ is an $(\mathfrak{l}_{\mathbb{C}},L\cap K)$-module, hence is also a $(\mathfrak{q},L\cap K)$-module by trivial extension to the radical $\mathfrak{u}$ of $\mathfrak{q}$.

Define the functor $\mathrm{pro}$ from $(\mathfrak{q},L\cap K)$-module to $(\mathfrak{g}_{\mathbb{C}},L\cap K)$-module as
\[\mathrm{pro}_{\mathfrak{q},L\cap K}^{\mathfrak{g}_{\mathbb{C}},L\cap K}(W)=\mathrm{Hom}_{\mathfrak{q}}(U(\mathfrak{g}_{\mathbb{C}}),W)_{L\cap K}.\]
Here the subscript ${L\cap K}$ means taking all the $L\cap K$ finite vectors.

Define the functor from $(\mathfrak{g}_{\mathbb{C}},L\cap K)$-module to $(\mathfrak{g}_{\mathbb{C}},K)$-module as
\[\Gamma^{\mathfrak{g}_{\mathbb{C}},K}_{\mathfrak{g}_{\mathbb{C}},L\cap K}(V)=\mathrm{Hom}_{R(\mathfrak{g}_{\mathbb{C}},L\cap K)}(R(\mathfrak{g}_{\mathbb{C}},K),V)_K.\]

It is left exact and has $i$-th right derivative functors, denoted by $(\Gamma^{\mathfrak{g}_{\mathbb{C}},K}_{\mathfrak{g}_{\mathbb{C}},L\cap K})^i$.

Let $(^{\mathrm{u}}\mathcal{R}_{\mathfrak{q},L\cap K}^{\mathfrak{g}_{\mathbb{C}},K})^i(W)$ denote
\[(\Gamma^{\mathfrak{g}_{\mathbb{C}},K}_{\mathfrak{g}_{\mathbb{C}},K\cap L})^i(\mathrm{pro}_{\mathfrak{q},L\cap K}^{\mathfrak{g}_{\mathbb{C}},L\cap K}(W)),\]
then the equation (\ref{pi_e}) can be written as
\[(^{\mathrm{u}}\mathcal{R}^{\mathfrak{g}_{\mathbb{C}},K}_{\mathfrak{q},G(\lambda_s)\cap K})^d(\tau\otimes \pi_n).\]

\subsubsection{Nilpotent step}

We focus on the special unipotent representations which are defined by Arthur, Barbasch and Vogan (see Adams-Barbasch-Vogan \cite{ABV}, chapter 27).

Let's introduce notation about partition firstly. Let $n\in \mathbb{Z}_+$, let $r\in\mathbb{Z}_+$, $k_i,l_i\in \mathbb{Z}_+$ such that $\sum_{1\leq i\leq r}k_il_i=n$ and $k_i>k_j$ for any $1\leq i<j\leq r$. We use $\{k_1^{l_1}\cdots k_r^{l_r}\}$ to denote the partition of $n$, \[\{\underbrace{k_1,\cdots,k_1}_{l_1},\cdots,\underbrace{k_r,\cdots,k_r}_{l_r}\}.\]
Let $P$ denote this partition, we will use $J_P$ to denote $\mathrm{diag}(J_{k_1}^{l_1},\cdots,J_{k_r}^{l_r})$,
use $J_P(a)$ to denote $\mathrm{diag}(J_{k_1}^{l_1}(a),\cdots,J_{k_r}^{l_r}(a))$, use $J_P(a,b)$ to denote $\mathrm{diag}(J_{k_1}^{l_1}(a,b),\cdots,J_{k_r}^{l_r}(a,b))$.

For the case of $G=G_n(\mathbb{R})$,
let $f=\mathrm{pr}(\xi)$ with $\xi=J_P$.
The special unipotent representation attached to $\mathcal{O}_f$ is
\[\pi=\mathrm{sgn}(\det)^{w_1}|_{G_{t_1}(\mathbb{R})}\times \cdots\times \mathrm{sgn}(\det)^{w_p}|_{G_{t_p}(\mathbb{R})},\]
where $\{t_{1},\cdots,t_{p}\}$ is the dual partition of $\{k_1^{l_1}\cdots k_r^{l_r}\}$ and $w_i=0,1$ for $1\leq i\leq p$.

For the case of $G=G_n(\mathbb{C})$,
let $f=\mathrm{pr}(\xi)$ with $\xi=J_P$.
The special unipotent representation attached to $\mathcal{O}_f$ is
\[\pi=1|_{G_{t_1}(\mathbb{C})}\times \cdots\times 1|_{G_{t_p}(\mathbb{C})}\]
where $\{t_{1},\cdots,t_{p}\}$ is the dual partition of $\{k_1^{l_1}\cdots k_r^{l_r}\}$.

In general, we have the following correspondence between coadjoint orbits and irreducible unitary representations.

For the case of $G=G_n(\mathbb{R})$.

When $n=2m$, assume that $P=\{k_1^{l_1}\cdots k_r^{l_r}\}$ is a partition of $m$.
Let $f\in \mathfrak{g}^*$ and $f=\mathrm{pr}(\xi)$ with $\xi=J_P(0,b/2)$, $b\in \mathbb{Z}_+$.

Firstly, we use the nilpotent step and attach the orbit $f_n=\mathrm{pr}(\xi_n)$, $\xi_n=J_P(0,0)$,
with the irreducible unitary representation $\pi_n$ of $G(f_s)\cong G_m(\mathbb{C})$
\[ 1|_{G_{t_1}(\mathbb{C})}\times \cdots\times 1|_{G_{t_p}(\mathbb{C})},\]
where $\{t_{1},\cdots,t_p\}$  is the dual partition of $\{k_1^{l_1}\cdots k_r^{l_r}\}$.

By direct calculation, we have
\[\tau=(\frac{\det}{|\det|})^{a+m}|_{G_m(\mathbb{C})} \ \text{and}\ e^{\rho(\mathfrak{u})}=(\frac{\det}{|\det|})^{m}|_{G_m(\mathbb{C})}.\]

Using Vogan \cite{V1}, Theorem 17.6, we have
\[(^{\mathrm{u}}\mathcal{R}^{\mathfrak{g}_{\mathbb{C}},K}_{\mathfrak{q},G(f_s)\cap K})^d(\tau\otimes \pi_n)=\mathop{\times}\limits_{i=1}^{p} \delta(t_i,b).\]

So the irreducible unitary representation $\mathop{\times}\limits_{i=1}^p\delta(t_i,b)$ is attached to the coadjoint orbit of $f=\mathrm{pr}(\xi)$ with $\xi=J_P(0,b/2)$.

When $f=\mathrm{pr}(\xi)$ with
\[\xi=\mathrm{diag}(J_{P_1}(0,b_1/2),\cdots,J_{P_r}(0,b_r/2),J_P),\]
where $r\in \mathbb{Z}_+$, $b_i\in \mathbb{Z}_+$, for $1\leq i\leq r$, $b_1>b_2>\cdots>b_r$, $P_i$ is a partition of $n_i$ and $P$ is a partition of $n_0$ with $\sum_{i=1}^{r}2n_i+n_0=n$.

Let $\{t_{i1},\cdots,t_{is_i}\}$ be the dual partition of $P_i$ for $1\leq i\leq r$. Let $\{t_{1},\cdots,t_{s}\}$ be the dual partition of $P$.

Let
\[H=G_{2n_1}(\mathbb{R})\times\cdots\times G_{2n_p}(\mathbb{R})\times G_{n_0}(\mathbb{R}),\]
then $G(f_s)\subset H$. Let $\mathfrak{h}$ be the Lie algebra of $H$ and $\mathfrak{h}_{\mathbb{C}}$ be the complexification of $\mathfrak{h}$. Let $\mathfrak{q}_1=\mathfrak{h}_{\mathbb{C}}\cap \mathfrak{q}$, and let $\mathfrak{u}_1$ be the nilpotent radical of $\mathfrak{q}_1$. Let $Q$ be the parabolic subgroup containing $H$ and the invertible upper triangular matrices. Set $d_1=\dim(\mathfrak{u}_1\cap\mathfrak{k})$.

We use Vogan \cite{V1}, Theorem 17.6, and get that
\begin{align*}
 (^{\mathrm{u}}\mathcal{R}^{\mathfrak{g}_{\mathbb{C}},K}_{\mathfrak{q},G(f_s)\cap K})^{d}(\tau\otimes \pi_n)
 &= \mathrm{Ind}_Q^G ((^{\mathrm{u}}\mathcal{R}^{\mathfrak{h}_{\mathbb{C}},K\cap H}_{\mathfrak{q_1},G(f_s)\cap K})^{d_1}(\tau\otimes \pi_n))
\\& = (\mathop{\times}\limits_{j=1}^{p} \mathop{\times}\limits_{i=1}^{s_j} \delta(t_{ji},b_j))   \times ( \mathop{\times}\limits_{j=1}^{s} \mathrm{sgn}(\det)^{w_j}|_{G_{t_j}(\mathbb{R})}) ,\end{align*}
where $w_j=0,1$ for any $1\leq j\leq s$.

\subsubsection{Representations of $\mathrm{GL}_n(k)$ attached to some kind of coadjoint orbits}

In general, for the case of $G_n(\mathbb{R})$,
let $f=\mathrm{pr}(\xi)$ with $\xi=\mathrm{diag}(\xi_1,\cdots,\xi_m)$, $m\in \mathbb{Z}_+$, and
\[\xi_i=\mathrm{diag}(J_{P_{i1}}(a_i,b_{i1}/2),\cdots,J_{P_{ip_i}}(a_i,b_{ip_i}/2),J_{P_i}(a_i)), 1\leq i\leq m,\]
where $p_i\in \mathbb{Z}_+$, $a_i\in \mathbb{R}$ for $1\leq i\leq m$, $a_1>\cdots>a_m$, $b_{ij}\in \mathbb{Z}_+$, and $b_{i1}>\cdots >b_{ip_i}$, $P_{ij}$ is a partition of $n_{ij}$ for $1\leq i\leq m$ and $1\leq j\leq p_i$, $P_i$ is a partition of $n_i$ for $1\leq i\leq m$, $n_{ij},n_i$ satisfy
\[\sum_{1\leq i\leq m}\sum_{1\leq j\leq p_i}2n_{ij}+\sum_{1\leq i\leq m}n_i=n.\]

Assume that $\{t_{ij,1},\cdots,t_{ij,s_{ij}}\}$ is the dual partition of $P_{ij}$, and $\{t_{i,1},\cdots,t_{i,s_i}\}$ is the dual partition of $P_i$.

The representation attached to $\mathcal{O}_f$ is $\pi=\pi_1\times\cdots\times\pi_m$, where
\[\pi_{l}=
|\det|^{ia_l}\otimes((\mathop{\times}\limits_{j=1}^{p_l}\mathop{\times}\limits_{k=1}^{s_{lj}} \delta(t_{lj,k},b_{lj}))\times (\mathop{\times}\limits_{k=1}^{s_l}\mathrm{sgn}(\det)^{w_{l,k}}|_{G_{t_{l,k}}(\mathbb{R})})),\]
with $w_{l,k}=0,1$ for any $1\leq l\leq m$, $1\leq k\leq s_l$.

For the case of $G=G_n(\mathbb{C})$,
let $f=\mathrm{pr}(\xi)$ with
\[\xi=\mathrm{diag}(J_{P_1}(a_1),\cdots,J_{P_m}(a_m)),\]
where $a_i\in\mathbb{R}$ for any $1\leq i\leq m$, and $a_1>\cdots>a_m$, $P_i$ is a partition of $n_i$ for $1\leq i\leq m$, and $n_i$($1\leq i\leq m$) satisfy $\sum_{1\leq i\leq m}n_i=n$.

The representation attached to  $\mathcal{O}_f$ is
\[\pi=\mathop{\times}\limits_{j=1}^{m}\mathop{\times}\limits_{k=1}^{s_{j}}|\det|^{ia_j}|_{G_{p_{jk}}(\mathbb{C})},\]
where $\{p_{j1},\cdots,p_{js_j}\}$ is the dual partition of $P_j$ for $1\leq j\leq m$.

\subsection{Representations attached to coadjoint orbits of $P_n(k)$}

Given an (almost) algebraic group $G$, Duflo \cite{D} constructed all the irreducible unitary representations of $G$ assuming the unitary dual of some reductive groups, basing on the Mackey's theory. It suggests that
the problem of attaching irreducible unitary representations with coadjoint orbits  can be reduced to the problem in the case of reductive groups. We summarize some results of Duflo \cite{D} without proofs in the followings.

\subsubsection{Coadjoint orbits of algebraic groups}

Duflo showed how to get all the coadjoint orbits of $\mathfrak{g}^*$ assuming that one knew the coadjoint orbits of reductive groups.

Let $\mathfrak{g}$ be the Lie algebra of $G$, and let $\mathfrak{g}^*$ be the dual of $\mathfrak{g}$.
For any $x\in \mathfrak{g}^*$, let
$G(x)$ be the stabilizer of $x$ in $G$ and let $\mathfrak{g}(x)$ be the Lie algebra of $G(x)$, let $^uG(x)$ be the unipotent radical of
$G(x)$ and let $^u\mathfrak{g}(x)$ be the Lie algebra of $^uG(x)$.

On $\mathfrak{g}/\mathfrak{g}(x)$, there is a natural symplectic form, denoted by $\omega_x$,
\[\omega_x(\overline{\alpha},\overline{\beta})=x([\alpha,\beta]),\]
where $\alpha,\beta\in \mathfrak{g}$ and $\overline{\alpha},\overline{\beta}$ denote the image of $\alpha$ and $\beta$ in $\mathfrak{g}/\mathfrak{g}(x)$.

Firstly, we need to define the linear forms of unipotent type (see \cite{D}, I). If $G$ is unipotent, all of $\mathfrak{g}^*$ are of unipotent type, and if $G$ is
reductive, $0$ is the unique element of $\mathfrak{g}^*$ of unipotent type. In general, its definition depends on the following things.

\begin{Definition} Let $x\in \mathfrak{g}^*$ and $\mathfrak{b}$ be a subalgebra of $\mathfrak{g}$. We say that
\begin{itemize}
\item[(i)]$\mathfrak{b}$ is coisotropic relative to $x$ (or coisotropic) if the orthogonal of $\mathfrak{b}$ in $\mathfrak{g}$ with respect to $\omega_x$, denoted by $\mathfrak{b}^{\perp}$, is contained in $\mathfrak{b}$.
\item[(ii)] $\mathfrak{b}$ is a polarization of $x$ if $\mathfrak{b}^{\perp}=\mathfrak{b}$.
\item[(iii)] An coisotropic subalgebra $\mathfrak{b}$ is of strongly
unipotent type if $\mathfrak{b}$ is algebraic (i.e. there exists algebraic subgroup with Lie subalgebra $\mathfrak{b}$) and $\mathfrak{b}=\mathfrak{g}(x)+$$^u\mathfrak{b}$.
\end{itemize}
\end{Definition}

\begin{Definition}
An element $x$ in $\mathfrak{g}^*$ is said to be of unipotent type if it satisfies
\begin{itemize}
\item[(i)] There exists a reductive factor of $\mathfrak{g}(x)$ contained in $\ker x$.
\item[(ii)] There is a subalgebra of strongly unipotent type relative to $x$.
\end{itemize}
\end{Definition}

Let $x\in \mathfrak{g}^*$ be a form of unipotent type, then we can construct a subalgebra of strongly unipotent type relative to $x$ by induction on $\dim \mathfrak{g}$ canonically (see \cite{D}, I.20). Let $u=x|_{^u\mathfrak{g}}$ and let $\mathfrak{h}=\mathfrak{g}(u)$. If $\mathfrak{h}=\mathfrak{g}$, then set $\mathfrak{b}=\mathfrak{h}$. If $\mathfrak{h}\neq \mathfrak{g}$, then $y=x|_{\mathfrak{h}}$ is a form of unipotent type. Let $\mathfrak{l}\subset \mathfrak{h}$ be the canonical subalgebra of strongly unipotent type relative to $y$, set $\mathfrak{b}=$$^u\mathfrak{g}+\mathfrak{h}$. Then $\mathfrak{b}$ is the canonical subalgebra of strongly unipotent type relative to $x$.

Let $D$ be the set of pairs $(x,\lambda)$, where $x$ is a linear form of unipotent type on
$\mathfrak{g}$, and $\lambda $ is an element in $L(x)$, $L(x)$ is the set of
linear forms over $\mathfrak{g}(x)$ whose restriction to $^u\mathfrak{g}(x)$ is equal to $x|_{^u\mathfrak{g}(x)}$.

Let $\mathfrak{r}$ be a reductive factor of $\mathfrak{g}(x)$, then the restriction map establish a bijection of $L(x)$ and $\mathfrak{r}^*$, sending $y\in L(x)$ to $y|_{\mathfrak{r}}$.

Let $(x,\lambda)\in D$. Let $\mathfrak{b}$ be a subalgebra of $\mathfrak{g}$, which is of unipotent type relative to
$x$ (see \cite{D},I.8 for definition, from \cite{D}, I.9, I.22, we have that when $x$ is of unipotent type, a subalgebra is of unipotent type relative to $x$ if and only if it is strongly unipotent type relative to $x$). Let $f$ be an element of $\mathfrak{g}^*$ whose restriction to $^u\mathfrak{b}$ is equal to $x|_{^u\mathfrak{b}}$, and
whose restriction restriction to $\mathfrak{g}(x)$ is equal to $\lambda$. Such an element $f$ exists by the definition of $L(x)$.

Let $G$ acts on $D$ naturally, then we have the followings.
\begin{Proposition}[\cite{D}, II.5]\label{po}
\begin{itemize}
\item[(i)] The orbit $G\cdot f$ in $\mathfrak{g}^*$ does not depend on the choices of $\mathfrak{b}$ and $f$. We denote it by $O_{x,\lambda}$.
\item[(ii)] The map $(x,\lambda)\mapsto O_{x,\lambda}$  induces a bijection from $D/G$ to $\mathfrak{g}^*/G$.
\end{itemize}
\end{Proposition}

\subsubsection{Construction of unitary dual of algebraic groups}

Duflo showed how to construct the irreducible unitary representations of $G$ using the coadjoint orbits and the irreducible unitary representations of some reductive groups.

The construction of representations used Metaplectic groups.

Since $G(x)$ acts on $\mathfrak{g}/\mathfrak{g}(x)$ and keeps the form $\omega_x$, we get a morphism
\[\phi: G(x)\to \mathrm{Sp}(\mathfrak{g}/\mathfrak{g}(x),\omega_x).\]
Let $\mathrm{Mp}(\mathfrak{g}/\mathfrak{g}(x),\omega_x)$ be the  metaplectic group, i.e., there is a nontrivial connected double cover
\[\varphi: \mathrm{Mp}(\mathfrak{g}/\mathfrak{g}(x),\omega_x)\to \mathrm{Sp}(\mathfrak{g}/\mathfrak{g}(x),\omega_x).\]

We define a double cover of $G(x)$ as
\[G(x)^{\mathfrak{g}}=\{(g,t)\ |\ g\in G(x), t\in \mathrm{Mp}(\mathfrak{g}/\mathfrak{g}(x),\omega_x)\ \text{and}\ \phi(g)=\varphi(t)\}.\]
Let $\psi: G(x)^{\mathfrak{g}}\to G(x)$ be the covering morphism, and let $(1,-1)$ denote the nontrivial element of $\ker \psi$.

Similarly, for any group $H$, which acts on a symplectic vector space $\mathfrak{m}$ and keep the symplectic form, we have a morphism $\phi': H\to \mathrm{Sp}(\mathfrak{m})$ and a double covering morphism $\varphi': \mathrm{Mp}(\mathfrak{m})\to \mathrm{Sp}(\mathfrak{m})$. Define a double cover
\[\widetilde{H}=\{(h,t)\ |\ h\in H, t\in \mathrm{Mp}(\mathfrak{m})\ \text{and}\ \phi'(h)=\varphi'(t)\}.\]

By Segal-Shale-Weil representation (see Duflo \cite{D}, II.6), any $\hat{s}\in \mathrm{Mp}(\mathfrak{m})$ which is mapped to $s\in \mathrm{Sp}(\mathfrak{m})$ can be represented by $(s,\theta_{\hat{s}})$ (or just denoted by $(s,\theta)$), where $\theta_{\hat{s}}$ is a function depended on $\hat{s}$, with complex value module $1$ over the set of  maximal isotropic subspaces in $\mathfrak{m}$. Therefore, any element $(h,t)\in \widetilde{H}$ ($h\in H, t\in  \mathrm{Mp}(\mathfrak{m})$) can be represented by $(h,\theta_{t})$. We will also use $H^{\mathfrak{m}}$ to denote $\widetilde{H}$.

Since $^uG(x)$ is unipotent, $^uG(x)$ can be embedded into $G(x)^{\mathfrak{g}}$. Let
\begin{align*}Y^{irr}(x)&=\{\text{irreducible unitary representation}\ \pi\\& \text{of}\  G(x)^{\mathfrak{g}}/^uG(x)\ |\ \pi((1,-1))=-1\}.\end{align*}
Let $C$ denote the set of
couples $(x,\tau)$, where $x$ is a linear form of unipotent type, and $\tau\in Y^{irr}(x)$. Then the group $G$ acts on $C$ naturally.

Now, we can construct an irreducible unitary representation $T_{x,\tau}$ of $G$ (see \cite{D}, III for details).

Let $\mathfrak{b}$ be a subalgebra of strongly unipotent type relative to $x$ which is stable under $G(x)$, and let $\mathfrak{v}$ denote the nilpotent radical of $\mathfrak{b}$. Let $V$ be the subgroup corresponding to $\mathfrak{v}$. Set $B=G(x)V$. Let $v$ denote the restriction of $x$ to $\mathfrak{v}$, and $G(x)^{\mathfrak{v}}$ is defined.

Let $\mathfrak{l}\subset \mathfrak{v}$ be a polarization relative to $v$, and let $L$ be the unipotent subgroup with Lie algebra $\mathfrak{l}$. Define the irreducible unitary representation of $V$ attached to $v$ as
\[T_v=\mathrm{Ind}_L^Ve^{iv|_{\mathfrak{l}}}.\]

Define
\[T_{x,\tau}=\mathrm{Ind}_B^G(\tau'\otimes S_vT_v),\]
where $\tau'$ and $S_v$ are defined as follows, and we have that $T_{x,\tau}$ is independent of the choice of $B$ (see \cite{D}, III.16).

Define $\tau'$ to be the representation of $G(x)^{\mathfrak{v}}$ as
\[\tau'(y,\psi)=(\varphi\psi^{-1})\overline{\tau}(y,\varphi),\]
here $\overline{\tau}=\tau\otimes \chi_{x}$ of $(G(x)^{\mathfrak{g}}/^uG(x))\rtimes$$^uG(x)$ $\cong G(x)^{\mathfrak{g}}$ with $\chi_{x}$ defined by $\mathrm{d}\chi_{x}=ix|_{^u\mathfrak{g}(x)}$,
and
\[\varphi\psi^{-1}=\varphi(\mathfrak{l}'+\mathfrak{g}^{\perp})\psi(\mathfrak{l}')^{-1}\]
is a constant which is independent on the choice of the maximal totally isotropic subspace $\mathfrak{l}'$ of $\mathfrak{v}$ (see \cite{D}, II.8).

Define $S_v$ to be the action of $G(x)^{\mathfrak{v}}$ on $T_v$ by
\[S_v(y,\varphi)=\varphi(\mathfrak{l})\cdot S'_{v,\mathfrak{l}}(y),\]
and $S'_{v,\mathfrak{l}}$ is  the action of $G(x)$ on $T_v$ defined by
\[S'_{v,\mathfrak{l}}(y)=||A(y)||^{-1}F_{\mathfrak{l},y\mathfrak{l}}A(y),\]
$F_{\mathfrak{l},y\mathfrak{l}}$ is the intertwining operator from $T_{yv}$ to $T_{v}$,
and $A(y)$ is the operator from $T_{v}$ to $T_{yv}$ with
\[A(y)\alpha(z)=\alpha(y^{-1}(z)), \alpha\in T_v, z\in V.\]
It can be checked that $S_v$ is independent of the choice of $\mathfrak{l}$ (see \cite{D}, II.10).

\begin{Proposition}[\cite{D}, III.12]\label{co}
For any element $(x,\tau)\in C$, we get an irreducible unitary representation $T_{x,\tau}$ of $G$. Moreover, the map $(x,\tau)\mapsto T_{x,\tau}$ induces a bijection from $C/G$ to $\widehat{G}$.
\end{Proposition}

Now, we can obtain the correspondence between coadjoint orbits and irreducible unitary representations of an algebraic group $G$ (see \cite{D}, III.19, III.20).

Let $f\in \mathfrak{g}^*$, by Proposition \ref{po}, we get $\mathcal{O}_f=\mathcal{O}_{x,\lambda}$ for some $x\in \mathfrak{g}^*$ of unipotent type and $\lambda\in L(x)\cong (\mathfrak{g}(x)/^u\mathfrak{g}(x))^*$. Assume that $\tau$ is an irreducible unitary representation of $G(x)^{\mathfrak{g}}/^uG(x)$ attached to $\lambda$ such that $\tau((1,-1))=-1$, then $T_{x,\tau}$ is an irreducible unitary representation of $G$ attached to $\mathcal{O}_f$.

\subsubsection{Duflo's construction in the $P_n(k)$ case}

Set $G=P_n(k)$, we show how to attach representations with coadjoint orbits by the above method (using the same notation).

For any element $f\in \mathfrak{g}^*$, applying the results in section 2, we have that $f$ is conjugated to an element $\mathrm{pr}'(\mathrm{diag}(A,J_m))$, for some $m\in \mathbb{Z}_+$ and $A\in \mathfrak{gl}_{n-m}(k)$.

Let $x=\mathrm{pr}'(\mathrm{diag}(0,J_m))$, then $x$ is of unipotent type.
Actually, $\mathfrak{g}(x)$ has reductive factor
\[\mathfrak{g}'=\{\mathrm{diag}(X,0_{m\times m})\ |\ X\in \mathfrak{gl}_{n-m}(k)\}\subset \ker(x).\]
Let
\[\mathfrak{b}=\{\left(\begin{array}{cc}X&Y \\ 0& Z\end{array}\right)\ |\ X\in \mathfrak{gl}_{n-m}(k), Y\in  M_{(n-m)\times m}(k), Z\in  \mathfrak{u}_m(k) \}, \]
where $\mathfrak{u}_m(k)$ is the set of strictly upper triangular matrices over $k$, then $\mathfrak{b}$ is a subalgebra of strongly unipotent type relative to $x$. Therefore, $x$ is of unipotent type.

Let $\lambda'=\mathrm{pr}'(\mathrm{diag}(A,0))$, then $\lambda:=\lambda'|_{\mathfrak{g}(x)} \in L(x)$, since $x|_{^u\mathfrak{g}(x)}=0$ and $\lambda|_{^u\mathfrak{g}(x)}=0$. Moreover,
\[L(x)\cong (\mathfrak{g}(x)/^u\mathfrak{g}(x))^*\cong\mathfrak{gl}_{n-m}(k)^*.\]
Also it is easy to see that $\mathcal{O}_f=\mathcal{O}_{x,\lambda}$.

Furthermore, we have that $G(x)^{\mathfrak{g}}$ is a trivial double cover of $G(x)$, that is $G(x)^{\mathfrak{g}}\cong G(x)\times \mathbb{Z}_2$. Actually,
\[G(x)=\{\left(\begin{array}{ccc}Y&Z&0 \\ 0&1&0\\ 0&0&I_{m-1}\end{array}\right)\ |\ Y\in \mathrm{GL}_{n-m}(k), Z\in  M_{(n-m)\times 1}(k)\}.\]
It is not hard to see that the image of the reductive factor
\[\{\mathrm{diag}(Y,I_m)\ |\ Y\in \mathrm{GL}_{n-m}(k)\}\]
in
\[\mathrm{Sp}(\mathfrak{g}/\mathfrak{g}(x),\omega_x)\cong \mathrm{Sp}((n-\frac{m}{2})(m-1),k)\]
has trivial double cover in the metaplectic group $\mathrm{Mp}((n-\frac{m}{2})(m-1),k)$. Therefore, $G(x)^{\mathfrak{g}}$ is a trivial double cover of $G(x)$.

To attach $\mathcal{O}_f$ with representation of $P_n(k)$,  we assume that $\tau_0$ is the irreducible unitary representation of
\[G(x)/^uG(x)\cong \mathrm{GL}_{n-m}(k)\]
which is attached to the $\mathrm{GL}_{n-m}(k)$-orbit of $\lambda\in L(x)\cong\mathfrak{gl}_{n-m}(k)^*$. Define an irreducible unitary representation of $G(x)/^uG(x)\times \mathbb{Z}_2$, denoted by $\tau$, such that $\tau|_{G(x)/^uG(x)}=\tau_0$ and $\tau((1,-1))=-1$. Then the irreducible representation of the irreducible unitary representation attached to $\mathcal{O}_f$ is
\[T_{x,\tau}=\mathrm{Ind}_B^G(\tau'\otimes S_vT_v).\]

The subalgebra of unipotent type relative to $x$ constructed above, $\mathfrak{b}$, is the canonical one. And the corresponding subgroup $B$ is
\[B=\{\left(\begin{array}{cc}Y & Z \\ 0& U \end{array}\right)\ |\ Y\in \mathrm{GL}_{n-m}(k), Z\in M_{(n-m)\times m}(k), U\in U_m(k)\}.\]
Here $U_m(k)$ is the set of upper triangular unipotent matrices over $k$. Moreover, $\mathfrak{v}=$$^u\mathfrak{b}$ itself is the polarization of $v=x|_{\mathfrak{v}}$, and we get a one dimensional unitary representation of $V$, $T_v=e^{ix|_{\mathfrak{v}}}$.

Since $\omega_x$ is trivial on $\mathfrak{b}$, we get $G(x)^{\mathfrak{v}}\cong G(x)\times \mathbb{Z}_2$ and $G(x)^{\mathfrak{b}}\cong G(x)\times \mathbb{Z}_2$, and therefore $S_v|_{G(x)}$ is the one dimensional trivial representation, so we have
\[\tau'\otimes S_vT_v=\tau_0\otimes e^{ix}.\]
Therefore, the irreducible unitary representation of $G$ attached to $\mathcal{O}_f$ with $f=\mathrm{pr}'(\mathrm{diag}(A,J_m))$ is
\[T_{x,\tau}=\mathrm{Ind}_B^G(\tau_0\otimes e^{ix})=\mathrm{I}^{m-1}\mathrm{E}\tau_0,\]
where $\tau_0$ is an irreducible unitary representation of $\mathrm{GL}_{n-m}(k)$ attached to the $\mathrm{GL}_{n-m}(k)$-orbit of $\mathrm{pr}_{n-m}(A)$.

\section{A generalization of Duflo's conjecture}\label{sec-6}

We establish a generalization of Duflo's conjecture for the restriction of an irreducible unitary representation $\pi$ of $G=\mathrm{GL}_n(k)$ to the mirabolic subgroup $P=P_n(k)$, $k=\mathbb{R}$ or $\mathbb{C}$, where $\pi$ is attached to a $G$-coadjoint orbit of $f\in \mathfrak{g}^*$ in section 5. Set $\mathcal{O}_{\pi}=\mathcal{O}_f$. Let $\mathrm{p}: \mathcal{O}_{\pi}\to \mathfrak{p}^*$ be the moment map. Then we have the following theorem.

\begin{Theorem}
There are only finitely many $P$-orbits in $\mathrm{p}(\mathcal{O}_{\pi})$, including a unique open $P$-orbit $\Omega$ in $\mathrm{p}(\mathcal{O}_{\pi})$. Moreover,
\begin{itemize}
\item[(1)] the moment map $\mathrm{p}: \mathcal{O}_{\pi}\to \mathfrak{p}^*$ is proper over $\Omega$,
\item[(2)] the restriction of $\pi$ to $P$, $\pi|_P$ is irreducible, and is attached to $\Omega$ in the sense of Duflo,
\item[(3)] the reduced space of $\Omega$ (with respect to the moment map $\mathrm{p}$) is a single point.
\end{itemize}
\end{Theorem}

\begin{proof}
By Theorem \ref{GLnC}, Theorem \ref{GLnR}, we get (1),(3) immediately. To get (2), we firstly get $\pi|_P$ from section 4 and get $\Omega=\mathrm{p}(g_{I^o}\cdot f)$ from the results of moment maps in section 3. It remains to verify if $\pi_P$ is attached to $\Omega$ by Duflo's construction in section 5. We verify one example and one can verify the other cases easily.

Let $ 1\leq t_p\leq \cdots\leq t_1\in \mathbb{Z}_+$ such that $\sum_{i=1}^pt_i=n$. Let $\pi$ be the unipotent representation
\[1|_{G_{t_1}(\mathbb{C})}\times \cdots\times 1|_{G_{t_p}(\mathbb{C})}\]
of $G_n(\mathbb{C})$, which is attached to the coadjoint orbit of $f=\mathrm{pr}(\xi)$ with
\[\xi=\mathrm{diag}(J_{k_1}^{l_1},\cdots,J_{k_r}^{l_r}),\]
where $k_i,l_i\in \mathbb{Z}_+$ for $1\leq i\leq r$, and $\{\underbrace{k_1,\cdots,k_1}_{l_1},\cdots,\underbrace{k_r,\cdots,k_r}_{l_r}\}$ is the dual partition of $\{t_1,\cdots,t_p\}$ (so $k_1\geq \cdots \geq k_r$).

Then
\[\pi|_P=I^pE(1|_{G_{t_1-1}(\mathbb{C})}\times \cdots\times 1|_{G_{t_p-1}(\mathbb{C})}).\]

To get the unique dense open $P$-orbit in $\mathrm{p}(\mathcal{O}_f)$, we set $\tilde{f}=\mathrm{pr}(\tilde{\zeta})$ with $\tilde{\zeta}=\mathrm{diag}(J_{k_r}^{l_r},\cdots,J_{k_1}^{l_1})$. Then $\tilde{f}$ is $G_n(\mathbb{C})$-conjugated to $f$, so $\mathcal{O}_f=\mathcal{O}_{\tilde{f}}$. By Theorem \ref{GLnC}, we have $\Omega=P\cdot \mathrm{p}(g_{I^o}\cdot\tilde{f})$ is the unique dense open $P$-orbit in $\mathrm{p}(\mathcal{O}_f)$.
Since $I^o=\{n\}$ and $g_{I^o}=I_n$, we have $g_{I^o}\cdot\tilde{f}=\tilde{f}$. It is easy to see that $\mathrm{p}(\tilde{f})$ is $P$-conjugated to the element $\mathrm{pr}'(\zeta)$ with
\[\zeta=\mathrm{diag}(J_{k_1}^{l_1-1},J_{k_2}^{l_2},\cdots,J_{k_r}^{l_r},J_{k_1}),\]
So $\Omega=P\cdot \mathrm{pr}'(\zeta)$.

Based on  Duflo's construction, we get that the irreducible unitary representation of $P$ attached to $\Omega$ is
\[I^{k_1}E(1|_{G_{s_1}(\mathbb{C})}\times \cdots\times 1|_{G_{s_q}(\mathbb{C})}),\]
where $\{s_1,\cdots,s_q\}$ is the dual partition of $\{\underbrace{k_1,\cdots,k_1}_{l_1-1},\underbrace{k_2,\cdots,k_2}_{l_2},\cdots,\underbrace{k_r,\cdots,k_r}_{l_r}\}$. By definition of dual partition, we have $p=q$, $s_i=t_i-1$ for $1\leq i\leq p$, and $k_1=p$. Hence we get that $\pi|_{P}$ is attached to $\Omega$.
\end{proof}

\section*{Acknowledgements}
The author would like to thank Jun Yu for suggesting this problem
and many discussions, thank Daniel Kayue Wong for discussions on the orbit method of reductive groups
and thank Huajian Xue for many discussions, etc.

\end{document}